\pdfoutput=1
\documentclass[12pt]{article}
\usepackage{color}
\usepackage{graphicx}
\usepackage{amsmath,amsfonts,amssymb,graphics,amsthm}
\usepackage{hyperref}
\usepackage{comment}
\usepackage{tabularx}
\usepackage[protrusion=true,expansion=true]{microtype}
\usepackage{enumerate}

\usepackage[margin=1.2in]{geometry}

\hypersetup{
    colorlinks=false,
    linktocpage,
    }

\numberwithin{equation}{section}

\newtheorem{theorem}{Theorem}[section]

\newtheorem{lemma}[theorem]{Lemma}
\newtheorem{proposition}[theorem]{Proposition}
\newtheorem{prop}[theorem]{Proposition} 
\newtheorem{lem}[theorem]{Lemma} 

\theoremstyle{remark}
\theoremstyle{remark}\newtheorem{defn}[theorem]{Definition}
\theoremstyle{remark}\newtheorem{remark}[theorem]{Remark}
\theoremstyle{remark}\newtheorem{notation}[theorem]{Notation}

\def\@rst #1 #2other{#1}
\newcommand\MR[1]{\relax\ifhmode\unskip\spacefactor3000 \space\fi
  \MRhref{\expandafter\@rst #1 other}{#1}}
\newcommand{\MRhref}[2]{\href{http://www.ams.org/mathscinet-getitem?mr=#1}{MR#2}}

\def\MR#1{\href{http://www.ams.org/mathscinet-getitem?mr=#1}{MR#1}}

\newcommand{\C}{\mathbb{C}}
\newcommand{\D}{\mathbb{D}}
\newcommand{\E}{\mathbb{E}}

\newcommand{\R}{\mathbb{R}}
\renewcommand{\P}{\mathbb{P}}
\newcommand{\bbH}{\mathbb{H}}

\renewcommand{\Im}{\mathrm{Im}}

\DeclareMathOperator{\diam}{diam}

\DeclareMathOperator{\SLE}{SLE}
\DeclareMathOperator{\CLE}{CLE}

\def\cW{\mathcal{W}}

\def\cS{\mathcal{S}}

\def\cG{\mathcal{G}}

\newcommand{\aryb}{\begin{eqnarray*}}
\newcommand{\arye}{\end{eqnarray*}}
\def\alb#1\ale{\begin{align*}#1\end{align*}}
\newcommand{\eqb}{\begin{equation}}
\newcommand{\eqe}{\end{equation}}
\newcommand{\eqbn}{\begin{equation*}}
\newcommand{\eqen}{\end{equation*}}

\newcommand{\BB}{\mathbb}
\newcommand{\ol}{\overline}

\newcommand{\op}{\operatorname}

\newcommand{\frk}{\mathfrak}
\newcommand{\eqD}{\overset{d}{=}}
\newcommand{\ep}{\epsilon}
\newcommand{\rta}{\rightarrow}

\newcommand{\wt}{\widetilde}
\newcommand{\wh}{\widehat} 
\newcommand{\mcl}{\mathcal}

\newcommand{\bdy}{\partial}

\DeclareMathAlphabet{\mathpzc}{OT1}{pzc}{m}{it}

\newcommand{\hwedge}{\mathpzc h}
\newcommand{\xwedge}{\mathpzc x}

\begin{document}

\author{
\begin{tabular}{c}Ewain Gwynne\\[-5pt]\small MIT\end{tabular}\;
\begin{tabular}{c}Nina Holden\\[-5pt]\small MIT\end{tabular}\;
\begin{tabular}{c}Jason Miller\\[-5pt]\small Cambridge\end{tabular}\;
\begin{tabular}{c}Xin Sun\\[-5pt]\small MIT\end{tabular}}

\title{Brownian motion correlation in the peanosphere~for~$\kappa >8$}
\date{}

\maketitle

\begin{abstract} 
The peanosphere (or ``mating of trees") construction of Duplantier, Miller, and Sheffield encodes certain types of $\gamma$-Liouville quantum gravity (LQG) surfaces ($\gamma \in (0,2)$) decorated with an independent $\SLE_{\kappa}$ ($\kappa = 16/\gamma^2 > 4$) in terms of a correlated two-dimensional Brownian motion and provides a framework for showing that random planar maps decorated with statistical physics models converge to LQG decorated with an $\SLE$.  Previously, the correlation for the Brownian motion was only explicitly identified as $-\cos(4\pi/\kappa)$ for $\kappa \in (4,8]$ and unknown for $\kappa > 8$.  The main result of this work is that this formula holds for all $\kappa > 4$.  This supplies the missing ingredient for proving convergence results of the aforementioned type for $\kappa > 8$.  Our proof is based on the calculation of a certain tail exponent for SLE$_{\kappa}$ on a quantum wedge and then matching it with an exponent which is well-known for Brownian motion.
\end{abstract}

\tableofcontents

\parindent 0 pt
\setlength{\parskip}{0.25cm plus1mm minus1mm}

\section{Introduction}
\label{sec:intro}

Suppose that $h$ is an instance of the Gaussian free field (GFF) on a planar domain $D$ and $\gamma\in(0,2)$. Formally the $\gamma$-Liouville quantum gravity (LQG) surface associated with $h$ is the Riemannian manifold with metric tensor given by
\begin{equation}
\label{eqn:lqg_metric}
e^{\gamma h(z)} (dx^2 + dy^2) ,
\end{equation}
where $dx^2 + dy^2$ denotes the Euclidean metric on $D$. This expression does not make literal sense since $h$ is a distribution and does not take values at points.  However, one can make sense of the volume form associated with~\eqref{eqn:lqg_metric} as a random measure via various regularization procedures, e.g.\ the ones used in \cite{shef-kpz}. The metric space structure of LQG has been constructed in the special case $\gamma=\sqrt{8/3}$ in \cite{lqg-tbm1} building on \cite{sphere-constructions} and, upon combining with \cite{tbm-characterization}, will be identified with the Brownian map in \cite{lqg-tbm2,lqg-tbm3}, but it remains an open problem to construct the metric for other values of $\gamma \in (0,2)$.

One of the main sources of significance of LQG is that it has been conjectured to describe the scaling limits of random planar maps decorated by statistical physics models.  This conjecture can be formulated in several different ways by specifying the topology. For example, one can view random planar maps as metric spaces and endow them with the Gromov-Hausdorff topology.  Convergence under this topology has been established in the case of uniformly random quadrangulations to the Brownian map in \cite{legall-uniqueness,miermont-brownian-map}. Combining with the aforementioned works gives the Gromov-Hausdorff convergence to $\sqrt{8/3}$-LQG. An alternative approach is to start off with a random planar map, embed it conformally into $\BB \C$ (e.g.\ via circle packing, Riemann uniformization, etc...) and show that the random area measure it induces (i.e., the pushforward of the uniform measure on the faces of the map) converges weakly to an LQG measure.  Establishing this type of convergence  is an open problem for any $\gamma\in (0,2)$. 

The work~\cite{wedges} takes a third approach through its \emph{peanosphere} or \emph{mating of trees} construction. More precisely, let $\gamma\in (0,2)$, $\kappa'=16/\gamma^2 >4$, and $(L_{t} ,R_{t} )_{t\in \R}$ be a correlated two-dimensional two-sided Brownian motion. Then $(L,R)$ encodes a pair of Brownian continuum random trees~\cite{aldous-crt1,aldous-crt2,aldous-crt3} with $L$ and $R$ as their contour functions. As explained in \cite[Section~1.1]{wedges}, one can glue the two trees together to obtain a topological sphere endowed with a measure and the space-filling peano curve which traces the interface between the two trees\footnote{This is the source of the name peanosphere.}. In \cite{wedges} the authors show that there is a canonical way of embedding this measure-endowed topological sphere into $\C\cup \{\infty\}$ such that the pushforward of the measure is a form of $\gamma$-LQG and the image of the spacing-filling curve is an independent space-filling form of Schramm's $\SLE$ \cite{schramm0} with parameter\footnote{We use the convention of~\cite{ig1,ig2,ig3,ig4} of writing $\kappa' > 4$ for the $\SLE$ parameter and $\kappa = 16/\kappa'$ for the dual parameter.} $\kappa'$ from $\infty$ to $\infty$ as defined in~\cite{ig4}; see also \cite{wedges}. Moreover, it is shown in \cite{wedges} that both the field $h$ and the space-filling $\SLE$ are a.s.\ determined by $(L,R)$.  That is, the peanosphere comes equipped with a canonical conformal structure.

It is proved in \cite{wedges} that for $\gamma\in[\sqrt{2},2)$ (equivalently, for $\kappa'\in(4,8]$) the correlation between $L$ and $R$ is given by $-\cos(4\pi/\kappa')\geq 0$.  The correlation between $L$ and $R$ for $\gamma\in(0,\sqrt{2})$ (equivalently, for $\kappa'>8$) is left as an open problem~\cite[Question~13.4]{wedges}. The main result of this paper is that the correlation between $L$ and $R$ is given by $-\cos(4\pi/\kappa')$ for all $\kappa' > 4$ (so that $L$ and $R$ are \emph{negatively} correlated for $\kappa' > 8$).

For $\kappa'\in (4,8]$, the peanosphere construction can be viewed as a continuum analogue of the bijection introduced by Sheffield in~\cite[Section~4.1]{shef-burger}, which encodes a critical Fortuin-Kasteleyn (FK) decorated planar map in terms of a word in a certain alphabet of five letters. Indeed, the manner in which the space-filling $\SLE_{\kappa'}$ path $\eta$ and the $\gamma$-LQG surface are encoded by $Z$ closely parallels the manner in which an FK planar map is described by a word under the bijection of~\cite{shef-burger} (see~\cite{wedges,gms-burger-cone,gwynne-miller-cle} for more details). This correspondence allows one to interpret various scaling limit statements for FK planar maps, as proven in~\cite{shef-burger,gms-burger-cone,gms-burger-local,gms-burger-finite}, as convergence results for FK decorated random planar maps to $\SLE$ decorated LQG with respect to the peanosphere topology; see also \cite{blr-exponents} for a calculation of some exponents associated with an FK planar map which match the corresponding exponents which can be derived in the continuum using \cite{wedges}.  Under this topology, two spanning tree decorated surfaces are said to be close if the contour functions of the tree/dual tree pairs are close.  On the FK planar map side, the tree/dual tree pair is generated using Sheffield's bijection~\cite{shef-burger} and in the continuum this pair is given by trees of GFF flow lines \cite{ig4} whose peano curve is space-filling $\SLE_{\kappa'}$.  This topology has been strengthened further using these constructions in \cite{gwynne-miller-cle}.

Recently the techniques of~\cite{shef-burger} have been generalized in~\cite{gkmw-burger} to the setting of random planar maps decorated with a certain type of spanning tree. It is in particular shown in \cite{gkmw-burger} that for a certain range of parameter values, the contour functions converge in the scaling limit to a \emph{negatively} correlated Brownian motion (which extends~\cite[Theorem~2.5]{shef-burger}). In another work~\cite{kmsw-bipolar}, it is shown that the height functions associated with the northwest tree and its dual tree which arise from a so-called bipolar orientation on a random planar map also converge to a certain pair of negatively correlated Brownian motions. The result of \cite{kmsw-bipolar} is strengthened (for the case of triangulations) in \cite{ghs-bipolar}, which shows convergence of two pairs of height functions to two pairs of negatively correlated Brownian motions, corresponding to two space-filling SLE curves traveling in a direction perpendicular (in the sense of imaginary geometry) to each other. In all of the above cases the correlation of the Brownian motion is explicit. Our main result allows us to interpret these limit results as convergence of random planar maps decorated with a statistical physics model to certain $\gamma$-LQG surfaces with $\gamma\in (0,\sqrt{2})$ decorated with an $\SLE_{\kappa'}$ with $\kappa' > 8$. 

Moreover, knowing the correlation of $(L,R)$ allows us to understand the interplay between two-dimensional Brownian motion and the space-filling $\SLE$ on top of the LQG surface at a quantitative level. For example,
 the KPZ-like formula established in~\cite{ghm-kpz} relates the Hausdorff dimension of an arbitrary random Borel set  $A\subset\BB C$ which is determined by the space-filling $\SLE_{\kappa'}$ (viewed modulo monotone reparameterization of time) in the peanosphere construction to the Hausdorff dimension of its pre-image under the Brownian motion $(L,R)$. This reduces the problem of computing the Hausdorff dimension of $A$ to the problem of computing the dimension of an (often much simpler) set defined in terms of $(L,R)$ (many examples of this type are given in \cite{ghm-kpz}). Our result implies that the formula derived in \cite{ghm-kpz} is valid for all $\kappa' > 4$ and not just $\kappa' \in (4,8]$.

Finally, we remark that our result supplies the missing ingredient in order to identify the correlation of the two-dimensional Brownian excursion appearing in the finite-volume version of the peanosphere construction~\cite[Theorem~1.1]{sphere-constructions} in the case $\gamma\in(0,\sqrt{2})$.

\bigskip

\noindent \textbf{Acknowledgements}
We thank Scott Sheffield for helpful discussions. We thank the Isaac Newton Institute for its hospitality during part of our work on this project. E.G.\ was supported by the U.S. Department of Defense via an NDSEG fellowship. N.H.\ was supported by a doctoral research fellowship from the Norwegian Research Council.  J.M.'s work was partially supported by DMS-1204894. X.S.\ was partially supported by NSF grant DMS-1209044.

\subsection{Main result}
\label{subsec:main}

 Now we give the formal statement of our main result. We will remind the  reader of the precise description of the objects involved in Section~\ref{subsec:pre}.
 
  Given $\gamma\in (0,2)$ and $\kappa'=16/\gamma^2$, let $\eta$ be a whole-plane space-filling $\op{SLE}_{\kappa'}$ from $\infty$ to $\infty$  (defined in~\cite[Sections 1.2.3 and 4.3]{ig4}; see also Section~\ref{sec:sle} of the present paper). Let $\gamma = 4/\sqrt{\kappa'}$ and let $\mcl C = (\BB C , h , 0, \infty)$ be a $\gamma$-quantum cone independent from $\eta$, as in~\cite[Section~4.3]{wedges} or Section~\ref{sec:surface} of the present paper. Let~$\mu_{h}$ and~$\nu_{h}$, respectively, be the $\gamma$-quantum area measure and $\gamma$-quantum boundary measure induced by $h$. Let $\wt{\eta} $ be the curve obtained by parameterizing $\eta$ by $\mu_h$-mass, so that $\wt{\eta} (0) =0$ and $\mu_{h}(\wt{\eta} ([t_1,t_2])) =t_2 - t_1$ for each $t_1 , t_2 \in \BB R$ with $t_1 <t_2$. 
 Let $Z_t = (L_t , R_t)$ denote the net change in the $\nu_{h}$-length of the left and right  boundaries of $\wt{\eta} ((-\infty,t])$ relative to time 0. Then $Z$ evolves as a two-sided Brownian motion with some correlation~\cite[Theorem~1.13]{wedges} and $Z$ a.s.\ determines the pair $(\eta , \mcl C)$ modulo rotation and scaling~\cite[Theorem~1.14]{wedges} (this is the mathematically precise formulation of the mating of trees/peanosphere construction described above). For $\gamma\in [\sqrt{2},2)$, by \cite[Theorem~1.13]{wedges} the correlation of $Z$ is $-\cos(4\pi/\kappa')$. Similar results are also proved in the upper half plane setting. See Section~\ref{sec:surface} for the definition of the quantum wedge.
  
 \begin{theorem}
 	\label{thm:main}
 	In the above setting, for $\gamma\in(0,\sqrt 2)$, the correlation  of $Z$ is still given by $  -\cos ( 4\pi/\kappa' ) $. Furthermore, suppose that $(\BB H , \hwedge, 0 , \infty)$ is a $3\gamma/2$-quantum wedge and~$\eta'$ is a chordal $\op{SLE}_{\kappa'}$ from $0$ to $\infty$ in $\BB H$ sampled independently of $\hwedge$ and let $\wt{\eta}'$ be the curve which arises by reparameterizing $\eta'$ by quantum mass with respect to $\hwedge$. Then the change in the left and right quantum boundary lengths of $\BB H\setminus \wt{\eta}'([0,t])$ with respect to $\hwedge$ evolve as a two-dimensional correlated Brownian motion with correlation $-\cos(4\pi/\kappa')$.
 \end{theorem}
 
We note that in light of Lemma~\ref{lem:wedge-cone} below, either of the two statements of Theorem~\ref{thm:main} implies the other.

\subsection{Preliminaries}\label{subsec:pre}

\subsubsection{Basic notation} 
\label{sec:basic}

Here we record some basic notation which we will use throughout this paper. 

\begin{notation}
\label{def:asymp}
If $a$ and $b$ are two quantities, we write $a\preceq b$ (resp.\ $a \succeq b$) if there is a constant $C$ (independent of the parameters of interest) such that $a \leq C b$ (resp.\ $a \geq C b$). We write $a \asymp b$ if $a\preceq b$ and $a \succeq b$. 
\end{notation}

\begin{notation}
\label{def:o-notation}
If $a$ and $b$ are two quantities which depend on a parameter $x$, we write $a = o_x(b)$ (resp.\ $a = O_x(b)$) if $a/b \rta 0$ (resp.\ $a/b$ remains bounded) as $x \rta 0$ or as $x\rta\infty$, depending on context. We write $a = o_x^\infty(b)$ if $a = o_x(b^s)$ for each $s > 0$ (if $b$ is tending to $0$) or for each $s < 0$ (if $b$ is tending to $\infty$). The regime we are considering will be clear from the context.
\end{notation}

Unless otherwise stated, all implicit constants in $\asymp, \preceq$, $\succeq$, $O_x(\cdot)$, and $o_x(\cdot)$ which are involved in the proof of a result are required to satisfy the same dependencies as described in the statement of said result.  

\subsubsection{Schramm-Loewner evolution}
\label{sec:sle}

Schramm-Loewner evolution (SLE$_{\kappa}$) is a one-parameter family of conformally invariant laws on two-dimensional fractal curves indexed by $\kappa > 0$, originally introduced in~\cite{schramm0} as a candidate for the scaling limit of various discrete statistical physics models. We refer the reader to~\cite{lawler-book,werner-notes} for an introduction to SLE.

\emph{Whole-plane space-filling SLE$_{\kappa'}$ from $\infty$ to $\infty$} for $\kappa' > 4$ is a variant of SLE$_{\kappa'}$ introduced in~\cite[Sections~1.2.3 and~4.3]{ig4} and~\cite[Footnote~9]{wedges}. In the case when $\kappa' \in (4,8)$, ordinary $\op{SLE}_{\kappa'}$ does not fill in open sets, but rather forms ``bubbles" which it surrounds, but never enters~\cite{schramm-sle}. Space-filling $\op{SLE}_{\kappa'}$ in this case is obtained by continuously filling in these bubbles as they are disconnected from $\infty$.  It is the peano curve associated with the exploration tree in the construction of $\CLE_{\kappa'}$ \cite{shef-cle}. In the case when $\kappa' \geq 8$ (so ordinary $\op{SLE}_{\kappa'}$ is space-filling) space-filling $\op{SLE}_{\kappa'}$ from $\infty$ to $\infty$ is a bi-infinite $\op{SLE}_{\kappa'}$ curve which fills in all of $\BB C$, starting and ending at $\infty$. It has the property that if one runs it up until any stopping time $\tau$, its complement is an unbounded simply connected domain and the conditional law of the path is given by that of an ordinary chordal $\SLE_{\kappa'}$ in the remaining domain from the tip at time $\tau$ to $\infty$.  It can also be constructed directly from ordinary $\SLE_{\kappa'}$ using a limiting procedure as follows (this is not equivalent to but easy to see from the GFF-based construction given in \cite{ig4}).  Suppose that $\eta'$ is a chordal $\SLE_{\kappa'}$ in $\BB H$ from $0$ to $\infty$ and that $z_0 \in \BB H$ is fixed.  For each $\epsilon > 0$ let $\eta_\epsilon'$ be given by $\epsilon^{-1}(\eta' - z_0)$ parameterized according to Lebesgue measure and, for each $r > 0$, let $\tau_{\epsilon,r}$ (resp.\ $\sigma_{\epsilon,r}$) be the first time that $\eta_\epsilon'$ hits $\partial B_r(0)$ (resp.\ fills $B_r(0)$).  Then the law of $\eta_\epsilon'|_{[\tau_{\epsilon,r},\sigma_{\epsilon,r}]}$ converges in total variation as $\epsilon \to 0$ to the restriction of whole-plane $\SLE_{\kappa'}$ from $\infty$ to $\infty$ to the interval of times between when it first hits $B_r(0)$ and fills $B_r(0)$, also parameterized according to Lebesgue measure.

We record the aforementioned fact about the conditional law of space-filling $\SLE_{\kappa'}$ for $\kappa' \geq 8$ in the following lemma, which is a consequence of the construction in~\cite[Footnote~9]{wedges}.

\begin{lemma}
\label{lem:wpsf-cond}
Let $\kappa' \geq 8$ and let $\eta$ be a whole-plane space-filling $\op{SLE}_{\kappa'}$ from $\infty$ to $\infty$. Let $\tau$ be a stopping time for $ \eta$. Then $\BB C\setminus \eta ((-\infty,\tau])$ is a.s.\ simply connected, unbounded, and the conditional law of $\eta |_{[\tau,\infty)}$ given $\eta |_{(-\infty,\tau]}$ is that of a chordal $\op{SLE}_{\kappa'}$ from $\eta (\tau)$ to $\infty$ in $\BB C\setminus \eta ((-\infty,\tau])$. 
\end{lemma}

\subsubsection{Quantum surfaces}
\label{sec:surface}

Fix $\gamma \in (0,2)$ (in this paper we will always take $\gamma = 4/\sqrt{\kappa'} \in (0,\sqrt 2)$). Also let $k$ be a non-negative integer. A \emph{$\gamma$-LQG surface} with $k$ marked points~\cite{shef-kpz,shef-zipper,wedges} is an equivalence class of $k+2$-tuples $(D , h , z_1,\dots ,z_k)$, where $D\subset \BB C$ is a domain (possibly all of $\BB C$), $h$ is a distribution on $D$, and $z_1,\dots ,z_k \in \ol{D}$ are marked points. Two such $k+2$-tuples $(D, h, z_1,\dots,z_k)$ and $(\wt D , \wt h,\wt z_1,\dots ,\wt z_k)$ are declared to be equivalent if there is a conformal map $f : \wt D \rta D$ such that 
\eqb \label{eq:lqg-coord}
\wt h = h \circ f + Q\log |f'| \quad \op{and} \quad f(\wt z_j) = z_j ,\: \forall j \in \{1,\dots,k\}  
\eqe
where 
\eqb \label{eq:Q-def}
Q := \frac{2}{\gamma} + \frac{\gamma}{2} .
\eqe 
In~\cite{shef-kpz}, it is shown that in the case when $h$ is some variant of the GFF on $D$ (which is the only case we will consider in this paper), the corresponding quantum surface has a natural area measure $\mu_h$ on $D$ (which is a regularization of ``$e^{\gamma h(z)} \, dz$") and a natural boundary length measure $\nu_h$ on $\bdy D$ (which is a regularization of ``$e^{\frac{\gamma}{2} h(z)} \, dz$"). By~\cite[Proposition~2.1]{shef-kpz} and its boundary analogue, these measures are preserved under transformations of the form~\eqref{eq:lqg-coord}. We note that the measure $\nu_h$ can be extended to certain curves lying in the interior of the domain $D$ (in particular, this is true for SLE$_{\kappa}$ curves with $\kappa = \gamma^2$). See~\cite{shef-zipper,wedges}. 

The main types of quantum surfaces which we will be interested in in this paper are the so-called quantum wedges and quantum cones, which are defined in~\cite[Section~1.6]{shef-zipper} and~\cite[Sections~4.2 and~4.3]{wedges}. For $\alpha \in (0,Q)$, an \emph{$\alpha$-quantum wedge} is a doubly marked quantum surface $\mcl W = (\BB H , \hwedge , 0, \infty)$ defined as follows. Let $\mcl H(\BB H)$ be the Hilbert space used to define a free-boundary GFF on $\BB H$~\cite[Section~3]{shef-zipper} (i.e.\ the completion of the space of smooth functions on $\BB H$ with respect to the inner product $(f,g)_\nabla = (2\pi)^{-1} \int_{\BB H} \nabla f(z) \cdot \nabla g(z) \, dz$).  
Let $\mcl H^0(\BB H)$ (resp.\ $\mcl H^\dagger(\BB H)$) be the space of functions in $\mcl H(\BB H)$ which are constant on each semicircle in $\BB H$ centered at 0 (resp.\ its orthogonal complement).  

Let $\alpha \in (0,Q)$, with $Q$ as in~\eqref{eq:Q-def}. Following~\cite[Definition 4.3]{wedges}, we define an $\alpha$-quantum wedge to be the doubly marked quantum surface $\mcl W = (\BB H , \hwedge , 0, \infty)$, where $\hwedge$ is a random distribution on $\BB H$ defined as follows. The projection $\hwedge^\dagger$ of $\hwedge$ onto $\mcl H^\dagger(\BB H)$ agrees in law with the projection onto $\mcl H^\dagger(\BB H)$ of a free-boundary GFF on $\BB H$. The projection $\hwedge^0$ of $\hwedge$ onto $\mcl H^0(\BB H)$ is independent of $\hwedge^\dagger$ and is defined as follows. For $s \geq 0$, $\hwedge^0(e^{-s}) = \mcl B_{2s} + \alpha s$, where $\mcl B$ is a standard linear Brownian motion; and for $s < 0$, $\hwedge^0(e^{-s}) = \wh{\mcl B}_{-2s} + \alpha s$, where $\wh{\mcl B}$ is independent from $\mcl B$ and has the law of a standard linear Brownian motion conditioned so that $\wh{\mcl B}_{2s } + (Q-\alpha) s  > 0$ for all $s > 0$. Note that a quantum wedge has two marked points, 0 and $\infty$. Every bounded subset of $\BB H$ has finite quantum mass a.s.\ and every neighborhood of $\infty$ (i.e.\ any open set which contains $\BB H \setminus B_r(0)$ for some $r>0$) has infinite mass a.s.
 
A quantum wedge is only defined modulo transformations of the form~\eqref{eq:lqg-coord}, so if we continue to parameterize the wedge by $(\BB H,0,\infty)$, the distribution $\hwedge$ can be replaced with another distribution obtained via~\eqref{eq:lqg-coord} with $f$ given by a scaling by a positive constant. Different choices of $h$ are referred to as different \emph{embeddings} of the same surface. 
 
\begin{defn}
\label{def:wedge-free}
The distribution $\hwedge$ defined just above is called the \emph{circle average embedding} of a quantum wedge. 
\end{defn}

We will consider several other embeddings of a quantum wedge in Section~\ref{subsec:Brownian-lower}.  

\begin{remark}
\label{rmk:wedge-free}
The circle average embedding of a quantum wedge is convenient for the following reason. Suppose that $h^F$ is a free-boundary GFF on $\BB H$ with additive constant chosen so that its circle average over $\bdy B_1(0) \cap \BB H$ is 0 (which is the main normalization used in~\cite{wedges}) and let $h := h^F - \alpha \log |\cdot|$. Then with $\hwedge$ as in Definition~\ref{def:wedge-free}, the restrictions of $\hwedge$ and $h$ to $B_1(0)\cap\BB H$ agree in law.  Indeed, if we let $h^0$ be the projection of $h$ onto $\mcl H^0(\BB H)$ (equivalently the semicircle average process around 0), then $h^0(e^{-s})$ evolves as a two-sided Brownian motion, so $(h^0(e^{-s}))_{s\geq 0} \eqD (\hwedge^0(e^{-s}))_{s\geq 0}$. Moreover, the projections of $\hwedge$ and $h$ onto $\mcl H^\dagger(\BB H)$ agree in law by definition. 
\end{remark}

For $\alpha \in (0,Q)$, an \emph{$\alpha$-quantum cone} is a doubly marked quantum surface $\mcl C = (\BB C , h , 0, \infty)$ which is similar to an $\alpha$-quantum wedge but is parameterized by the whole plane rather than the half plane. We will now describe the definition of this object, which first appeared in~\cite[Definition 4.9]{wedges}. Let $\mcl H(\BB C)$ be the Hilbert space used to define the whole-plane GFF on $\BB C$. Let $\mcl H^0(\BB C)$ (resp.\ $\mcl H^\dagger(\BB C)$) be the space of functions in $\mcl H(\BB C)$ which are constant on each circle centered at 0 (resp.\ its orthogonal complement). An embedding $h$ of a $\gamma$-quantum cone into $\BB C$ can be constructed as follows. The projection $h^\dagger$ of $h$ onto $\mcl H^\dagger(\BB C)$ agrees in law with the corresponding projection of a whole-plane GFF on $\BB C$. The projection $ h^0$ of $ h$ onto $\mcl H^0(\BB C)$ is independent of $ h^\dagger$ and is described as follows. For $s \geq 0$, $  h^0(e^{-s}) = \mcl B_{ s} + \alpha s$, where $\mcl B$ is a standard linear Brownian motion; and for $s < 0$, $  h^0(e^{-s}) = \wh{\mcl B}_{- s} + \alpha s$, where $\wh{\mcl B}$ is a standard linear Brownian motion conditioned so that $\wh{\mcl B}_{ s } + (Q-\alpha) s  > 0$ for all $s > 0$, independent from $\mcl B$. 

\begin{remark}
In~\cite{wedges}, the sets of quantum cones and quantum wedges are sometimes parameterized by a different parameter, called the \emph{weight}, which is equal to $\gamma(\gamma/2+Q-\alpha)$ in the wedge case and $2\gamma(Q-\alpha)$ in the cone case, with $Q$ as in~\eqref{eq:lqg-coord}. The reason for this choice of parameter is that it behaves nicely under the various ``gluing" and ``cutting" operations considered in~\cite{wedges}. In this paper we will not consider the weight parameter and will always identify our wedges and cones by $\alpha$, the size of the logarithmic singularity at 0. 
\end{remark}

The main fact which we will use about quantum cones in this paper is the following lemma, which allows us to reduce the problem of studying a space-filling $\op{SLE}_{\kappa'}$ on a $\gamma$-quantum cone to the problem of studying an ordinary chordal $\op{SLE}_{\kappa'}$ on a $\frac32\gamma$-quantum wedge.

\begin{lemma}
\label{lem:wedge-cone}
Let $\kappa' \geq 8$ and $\gamma=  4/\sqrt{\kappa'} \in (0,\sqrt 2]$. 
Let $\mcl C = (\BB C , h , 0 , \infty)$ be a $\gamma$-quantum cone. Let $\eta $ be a whole-plane space-filling $\op{SLE}_{\kappa'}$ from $\infty$ to $\infty$ independent from $\mcl C$. Let $\wt{\eta} $ be the curve obtained by parameterizing $\eta $ by $\gamma$-quantum mass with respect to $h$ so that $\wt{\eta} (0) = 0$. Let $\mcl W$ be the quantum surface obtained by restricting $h$ to $\BB C\setminus \wt{\eta} ((-\infty , 0])$. Then the pair $(\mcl W , \wt{\eta} |_{[0,\infty) }  )$ has the law of a $\frac{3\gamma}{2}$-quantum wedge together with an independent chordal $\op{SLE}_{\kappa'}$ parameterized by quantum mass with respect to this wedge. 
\end{lemma}
\begin{proof}
This is essentially proven as part of the proof of~\cite[Lemma~9.2]{wedges}, but we give the details for completeness.
Let $\eta_- $ and $\eta_+ $ be the left and right boundaries of $\wt{\eta} ((-\infty, 0 ])$. Then $\eta_\pm $ are independent of $\mcl C$; the law of $\eta_- $ is that of a whole-plane $\op{SLE}_\kappa(2-\kappa)$ from 0 to $\infty$; and the conditional law of $\eta_+ $ given $\eta_- $ is that of a chordal $\op{SLE}_\kappa(-\kappa/2;-\kappa/2   )$ from 0 to $\infty$ in $\BB C \setminus \eta_- $. Indeed, this follows from the construction of~\cite[Footnote 9]{wedges} as well as~\cite[Theorems 1.1 and 1.11]{ig4}.  
By~\cite[Theorem~1.12]{wedges}, the law of the quantum surface $\mcl W'$ obtained by restricting $\mcl C$ to $\BB C\setminus \eta_- $ is that of a $(2\gamma - 2/\gamma)$-quantum wedge. 
By~\cite[Theorem~1.9]{wedges}, the surface $\mcl W$ obtained by cutting $\mcl W'$ by $\eta_+ $ has the law of a $\frac{3\gamma}{2}$-quantum wedge. 
 The law of $\wt{\eta} |_{[0,\infty)}$ is obtained from Lemma~\ref{lem:wpsf-cond} and the independence of $\mcl W$ and $\wt{\eta} |_{[0,\infty)}$ (the latter viewed as a curve modulo monotone reparameterization) follows from independence of $\eta $ and $\mcl C$ together with Lemma~\ref{lem:wpsf-cond}.
\end{proof}

\subsection{Approximate cone time event}
\label{sec:cone-time}

In this subsection we reduce the proof of Theorem~\ref{thm:main} to the problem of calculating the tail exponent for the probability of a certain event. 

Assume we are in the setting described in Section~\ref{subsec:main}.
A \emph{$\pi/2$-cone time} of the Brownian motion $Z = (L,R)$ is a time $t\in\BB R$ for which there exists $t' > t$ such that $L_s \geq L_t$ and $R_s \geq R_t$ for each $s\in [t,t']$. That is, $Z$ stays in the ``cone" $\R_+^2 + Z_t$ for some positive amount of time after $t$. 
In the case when $\kappa' \in (4,8)$, the covariance of the peanosphere Brownian motion $Z$ is obtained by computing the Hausdorff dimension of its $\pi/2$-cone times in terms of $\kappa'$ and comparing the formula thus obtained to the known formula for the Hausdorff dimension of the set of $\pi/2$-cone times in terms of the correlation~\cite{evans-cone}. The Hausdorff dimension is calculated in terms of $\kappa'$ by observing that cone times for the Brownian motion correspond to local cut times for $\eta $, see \cite[Lemma~9.4]{wedges}.

In the case when $\kappa' >8$, the curve $\eta $ a.s.\ does not have any local cut times, so $Z$ a.s.\ does not have any $\pi/2$-cone times, hence has non-positive correlation~\cite{shimura-cone}. To compute the correlation in this case, we will compute the tail exponent for the probability that 0 is an ``approximate $\pi/2$-cone time" for $Z$, meaning that the event
\eqb \label{eq:def-E}
\wt E_\delta^t := \left\{ \inf_{s \in [0,t]} L_s \geq -\delta \:  \op{and} \: \inf_{s \in [0,t]} R_s \geq -\delta\right\}
\eqe
occurs for $t$ close to 1 and $\delta$ close to $0$. The tail exponent for probability of $\wt E^t_\delta$ is computed in terms of the correlation of $L_t$ and $R_t$ in~\cite[Equation (4.3)]{shimura-cone}. 

\begin{lemma}\label{lem:cone-exponent}
Let $-\alpha = -\alpha(\gamma)$ be the correlation of $L$ and $R$ and let
\eqb \label{eq:sigma-def}
\sigma(\gamma) := \frac{\pi}{\arccos(\alpha)} .
\eqe
There is a constant $c > 0$, depending only on $\alpha$, such that for $\delta>0$ and $t \geq \delta^{1/2}$ we have
\eqbn
\P\!\left[\wt E^t_\delta \right] = (c + o_\delta(1)) t^{- \sigma(\gamma) /2}   \delta^{ \sigma(\gamma) },
\eqen
where here the $o_\delta(1)$ is uniformly bounded for $\delta>0$ and $t \geq \delta^{1/2}$ and tends to 0 as $\delta\rta 0$ for each fixed $t$. 
\end{lemma}
\begin{proof}
Let $A$ be a linear transformation chosen in such a way that $\wt Z := A Z$ is a standard two-dimensional Brownian motion (variances equal to 1, covariance equal to 0). Then an approximate $\pi/2$-cone time for $Z$ is the same as an approximate $\arccos(\alpha)$-cone time for $\wt Z$. Hence the statement of the lemma follows from~\cite[Equation (4.3)]{shimura-cone}.
\end{proof}

In light of Lemma~\ref{lem:cone-exponent}, to prove Theorem~\ref{thm:main} it suffices to show that
\eqb \label{eq:exponent-compare}
\sigma(\gamma) =  \frac{4}{\gamma^2} =  \frac{\kappa'}{4}  .
\eqe

\begin{remark} \label{rmk:E-equiv}
The event $\wt E^t_\delta$ of~\eqref{eq:def-E} can equivalently be defined as follows. Let $f : \BB C\setminus \wt{\eta} ((-\infty, 0]) \rta \BB H$ be a conformal map which takes 0 to 0 and $\infty$ to $\infty$. Let $\hwedge := h \circ f^{-1} + Q\log |(f^{-1})'|$ and let $\wt{\eta}' := f(\wt{\eta} |_{[0,\infty)})$. By Lemma~\ref{lem:wedge-cone}, the quantum surface $\mcl W = (\BB H , \hwedge , 0 ,\infty)$ has the law of a $\frac32\gamma$-quantum wedge and $\wt{\eta}'$ is a chordal $\op{SLE}_{\kappa'}$ from $0$ to $\infty$ in $\BB H$ which is independent from $\mcl W$ and parameterized by $\gamma$-quantum mass with respect to $\hwedge$. For $\delta>0$, let $\xwedge_{\delta,L}$ and $\xwedge_{\delta,R}$ be the unique points respectively in $\BB R_-$ and $\BB R_+$  so that $\nu_{\hwedge}([-\xwedge_{\delta,L} , 0]) = \nu_{\hwedge}([0,\xwedge_{\delta,R}]) = \delta$. Then $\wt E^t_\delta$ is the same as the event that $\wt{\eta}'$ does not hit either $(-\infty,-\xwedge_{\delta,L}]$ or $[\xwedge_{\delta,R},\infty)$ before time~$t$. Since $\wt{\eta}'$ is boundary filling, $\wt E^t_\delta$ is also the same as the event that $\wt{\eta}'$ does not hit either $-\xwedge_{\delta,L}$ or $\xwedge_{\delta,R}$ before time~$t$.
\end{remark}

\subsection{Outline}
\label{sec:outline}

In the remainder of this paper, we will prove~\eqref{eq:exponent-compare}, hence Theorem~\ref{thm:main}.  For the proof, we will use the alternative description of the event $\wt E^t_\delta$ given in Remark~\ref{rmk:E-equiv}. In Section~\ref{sec:Euclidean}, we will use the SLE martingales of~\cite{sw-coord} to prove an estimate for the probability that a chordal $\op{SLE}_{\kappa'}$ from 0 to $\infty$ in $\BB H$ exits the Euclidean ball of fixed radius $r > 0$ before hitting $-z_L$ or $z_R$, where $z_L , z_R \in (0,\infty)$. In Section~\ref{sec:KPZ}, we will prove some moment estimates for the quantum boundary measure induced by a GFF which together with the estimates of Section~\ref{sec:Euclidean} will enable us to prove a variant of~\eqref{eq:exponent-compare} with $\wt E^t_\delta$ replaced by the event that the following is true. With $\xwedge_{\delta,L}$ and $\xwedge_{\delta,R}$ as in Remark~\ref{rmk:E-equiv}, the curve $\wt{\eta}'$ exits the Euclidean ball of radius $r$ before hitting either $-\xwedge_{\delta,L}$ or $\xwedge_{\delta,R}$. The arguments of this section are similar to those used to estimate the quantum measure in~\cite[Section~4]{shef-kpz}. In Section~\ref{sec:embedding}, we will extract~\eqref{eq:exponent-compare} from the estimate of Section~\ref{sec:KPZ}. In particular, we will prove that the probability of $\wt E^t_\delta$ with $t = \delta^{o_\delta(1)}$ is not too much different than the probability of the event of Section~\ref{sec:KPZ} for an appropriate choice of $r$. We will use some techniques which are similar to those found in~\cite[Section~10.4]{wedges}

\section{Euclidean exponent for the SLE event}
\label{sec:Euclidean}

Recall that $\eta'$ is an $\SLE_{\kappa'}$ from 0 to $\infty$ in $\bbH$, and let $(W_t)_{t\geq 0}$ and $(g_t)_{t\geq 0}$ denote its Loewner driving function and Loewner maps, respectively. Assume throughout this section that $\eta'$ is parameterized by half-plane capacity, and let
\eqb \label{eq:sle-filtration}
\mcl F_t := \sigma(\eta'(s) : s \in [0,t]).
\eqe
The purpose of this section is to prove the following proposition, i.e., we calculate the exponent for a Euclidean analogue of the event $\wt E^t_\delta$ of~\eqref{eq:def-E}.

\begin{prop}
\label{prop:Euclidean}
For any $T>0$ and $z_L,z_R\in(0,1)$ define the event $E^T_{z_L,z_R}$ by
\eqbn
E^T_{z_L,z_R} := \{ -z_L,z_R\not\in \eta'([0,T])\}.
\eqen
Then the following estimate holds for $\rho:=\kappa'-4$:
\eqb
\begin{split}
\P\!\left[E^T_{z_L,z_R}\right] &= \P\!\left[E^1_{z_L/\sqrt T,z_R/\sqrt T}\right],\,\,\,\\ \P[E^1_{z_L,z_R}]&=(z_L+z_R)^{\rho^2/(2\kappa')}
z_L^{\rho/\kappa'+o_{z_L}(1)}
z_R^{\rho/\kappa'+o_{z_R}(1)},
\end{split}
\label{eq:eucl1}
\eqe
where the rates of convergence of $o_{z_L}(1)$ and $o_{z_R}(1)$ depend only on $\kappa'$. For any $r>0$ define the stopping time $T_r:=\inf\{t>0\,:\,|\eta'(t)|\geq r\}$. Then 
\eqb
\begin{split}
\P\!\left[E^{T_r}_{z_L,z_R}\right] &= \P\!\left[E^{T_1}_{z_L/r,z_R/r}\right],\,\,\,\\
\P[E^{T_1}_{z_L,z_R}] &= (z_L+z_R)^{\rho^2/(2\kappa')}
z_L^{\rho/\kappa'+o_{z_L}(1)}
z_R^{\rho/\kappa'+o_{z_R}(1)}.
\end{split}
\label{eq:eucl2}
\eqe
\label{prop:Euclidean0}
\end{prop}

Both for the upper and the lower bound in Proposition~\ref{prop:Euclidean0} we will use the following result from \cite[Theorem~6, Remark~7]{sw-coord}. 
\begin{lem}
Define $\rho:=\kappa'-4$, and let $\wh T_L$ (resp.\ $\wh T_R$) denote the first time that $\eta'$ hits $-z_L$ (resp.\ $z_R$). For each $t\in [0,\wh T_L]$ (resp.\ $t\in[0,\wh T_R]$) define $z_t^L:=g_t(-z_L)$ (resp.\ $z_t^R:=g_t(z_R)$) and define the stochastic process $(M_t)_{t\geq 0}$ by
\eqbn
M_t= 
\begin{cases}
|W_t-z^R_t|^{\rho/\kappa'} |W_t-z^L_t|^{\rho/\kappa'} |z^R_t-z^L_t|^{\rho^2/(2\kappa')} & \text{if } t\in[0,\wh T_L\wedge \wh T_R],\\
0 & \text{if } t\geq \wh T_L\wedge \wh T_R.
\end{cases}
\eqen
Then $M_t$ is a local martingale.
\label{thm:martingale}
\end{lem}
It is also proved in \cite[Theorem~6]{sw-coord} that the law of $\eta'$ weighted by $M_t$ (run up to an appropriate stopping time) has the law of a chordal SLE$_{\kappa'}(\rho ; \rho)$ with force points at $-z_L$ and $z_R$, but we will not need this result. Note that the derivative term in \cite{sw-coord} vanishes for $\rho=\kappa'-4$.

For our proof of the lower bound in Proposition~\ref{prop:Euclidean0} we will need that $(M_t)_{t\geq 0}$ is a true martingale, not only a local martingale.
\begin{lem}
	The local martingale $(M_t)_{t\geq 0}$ defined in Lemma~\ref{thm:martingale} is martingale.
	\label{lem:eucl3}
\end{lem}
\begin{proof}
It is sufficient to prove that for any $t\geq 0$ we have $\E[\sup_{s\in[0,t]} M_s]<\infty$. This is sufficient, since if $(\sigma_k)_{k\in\BB N}$ are stopping times such that $(M_{\sigma_k\wedge t})_{t\geq 0}$ are martingales and $\sigma_k\rta\infty$ a.s., we can use the dominated convergence theorem to argue that $M_{\sigma_k\wedge t}\rta M_t$ in $L^1$ for each $t\geq 0$. 

Applying It\^o's formula and the Loewner equation, we have that both $z^R_t-W_t$ and $W_t-z^L_t$ are constant multiples of Bessel processes of dimension $1+\tfrac{4}{\kappa'} < 2$.  Since the law of a Bessel process of dimension $\delta$ is stochastically dominated by the law of a Bessel process of dimension $\delta'$ provided $0 < \delta < \delta'$, it follows that there exist two stochastic processes $\wh B^R,\wh B^L$ which are constant multiples of two-dimensional Bessel processes such that $z^R_t-W_t\leq \wh B^R_t$ and $z^L_t-W_t\leq \wh B^L_t$ for all $t\geq 0$. Since $\wh B^R$ and $\wh B^L$ each have the law of the modulus of a two-dimensional Brownian motion, Doob's maximal inequality implies that 
\eqb
\begin{split}
	\P\!\left[\sup_{s\in[0,t]} | z^q_s-W_s | >x\right] =o_x^\infty(x) \quad\text{for}\quad q \in \{L,R\}.
\end{split}
\label{eq:eucl9}
\eqe
Using that $|z_t^R - z_t^L| \leq 2 \max(|z_t^L - W_t|,|z_t^R - W_t|)$, $\rho > 0$ (so that all of the exponents in the definition of $M_t$ are positive), and the sum of the exponents in the definition of $M_t$ is equal to $\tfrac{\kappa'}{2}-2$, we have that
\eqbn
\E\!\left[\sup_{s\in[0,t]} M_s\right]
\preceq \E\!\left[\sup_{s\in[0,t]} (|z_s^R-W_s|\vee 1)^{\kappa'/2-2}\right]
+\E\!\left[\sup_{s\in[0,t]} (|z_s^L-W_s|)\vee 1)^{\kappa'/2-2}\right]
<\infty.
\eqen
\end{proof}

We have $M_0=(z_L+z_R)^{\rho^2/(2\kappa')}
z_L^{\rho/\kappa'}
z_R^{\rho/\kappa'}$, i.e., $M_0$ has the same exponents as the probability of the event $E^1_{z_L,z_R}$ in Proposition~\ref{prop:Euclidean0}. In order to prove the estimate~\eqref{eq:eucl1} for $T=1$ it is therefore sufficient to prove that $\P[E^1_{z_L,z_R}]$ is approximately equal to the expected value of $(M_t)_{t\geq 0}$ stopped at some appropriate stopping time. This is our strategy for the proof both of the upper bound and of the lower bound in~\eqref{eq:eucl1}.

\subsection{Euclidean upper bound}
\label{sec:upper-bound-Euc}
As indicated above we will establish the upper bound of $\P[E^1_{z_L,z_R}]$ in Proposition~\ref{prop:Euclidean0} by defining an appropriate stopping time for $(M_t)_{t\geq 0}$. We will use the following stopping time $\sigma_u$ for some $u>0$:
\eqb
\sigma_u=\inf \{t\geq 0: M_t = (z_Rz_L)^u \;\textrm{or}\;M_t=0\}.
\label{eq:eucl4}
\eqe
In order to prove that $\P[E^1_{z_L,z_R}]$ is bounded above by $\E[M_{\sigma_u}]=\E[M_0]$ (up to $o(1)$ errors) it is sufficient to prove that $\sigma_u<1$ with very high probability for small $z_L,z_R$. This is sufficient since $\E[M_{\sigma_u}]$ is approximately equal to $\P[M_{\sigma_u}>0]$ for small $u$. The following two technical lemmas will help us establish that $\sigma_u<1$ with very high probability. In Lemma~\ref{lem:eucl1} we prove that with very high probability $\op{Im}\eta'(t)$ does not stay close to the real line for all $t\in[0,1]$. Lemma~\ref{lem:eucl2} implies a lower bound for $M_t$ in terms of $\op{Im}\eta'(t)$.

\begin{lem}
For each $\ep>0$ we let
\eqbn
F_\ep:= \left\{ \sup_{t \in [0,1]} \Im (\eta'(t))\le \ep \right\}.
\eqen
Then $\P[F_{\ep}]=o^\infty_\ep(\ep)$.
\label{lem:eucl1}
\end{lem}
\begin{proof}
	By scale invariance of SLE the statement of the lemma is equivalent to the statement that if $F'_n:=\{ \sup_{t \in [0,n]} \Im (\eta'(t))\le 1\}$ for $n\in\BB N$ then we have $\P[F'_{n}]=o^\infty_n(n)$. Define the stopping time $\wt T$ by 
	\eqb
	\wt T:=\inf\{t\geq 0\,:\, \op{Im}(\eta'(t))\geq 1 \}. 
	\label{eq:eucl5}
	\eqe
	It is sufficient to prove that there is a constant $p>0$ s.t.\ for all $n\in\BB N$\ 
	\eqb
	\P[\wt T<(n+1)\,|\,\mcl F_{n}]\geq p,
	\label{eq:eucl3}
	\eqe 
	since this bound implies $\P[F'_n]\leq (1-p)^{n}=o_n^\infty(n)$. Here $\mcl F_n$ is as in~\eqref{eq:sle-filtration}.
	
	Define $p:=\P[\wt T<1]$. Since~\eqref{eq:eucl3} clearly holds on the event that $\wt T\leq n$ we assume $\wt T>n$. Define $l:=\{x+i\,:\,x\in\R\}$, and for each $n\in\BB N$ define $l'_n:=\{g_n(z)\,:\,z\in l\}$. By \cite[Equation (4.5)]{lawler-book} each $z\in l'_n$ satisfies $\op{Im}(z)\leq 1$, and $l'_n$ is a connected set dividing the upper half plane into an upper and a lower part, hence the SLE $g_n(\eta')$ hits $l'_n$ before it hits $l$. The estimate~\eqref{eq:eucl3} follows by the conformal Markov property of $\SLE$.
\end{proof}

\begin{lem}
Consider the stopping time $\wt T$ defined by~\eqref{eq:eucl5}. Let $S\subset\BB H$ denote the right boundary of $\eta'([0,\wt T])$, and let $\lambda$ denote Lebesgue measure on $\BB R$. Then there is a universal constant $c>0$ such that $\lambda(g_{\wt T}(S)) >c$.  The same likewise holds if we instead take $S$ to be the left boundary of $\eta'([0,\wt{T}])$.
\label{lem:eucl2}
\end{lem}
\begin{proof}
We assume without loss of generality that $S$ is equal to the right boundary of $\eta'([0,\wt T])$.  Let $(\mcl B_t)_{t\geq 0}$ be a Brownian motion in $\BB C$ independent of $\eta'$, and for each $z\in\BB H$ let $\P^z[\cdot]$ be the law under which $\mcl B_0=z$. For each $z\in\BB H$ let $I_z$ be the horizontal line segment from $i+z-1$ to $i+z+1$, and define the two stopping times $\ol\tau$ and $\wh\tau$ for $(\mcl B_t)_{t\geq 0}$ (conditioned on $\mcl F_{\wt T}$) as follows.
\eqbn
\begin{split}
\ol\tau = \inf\{t\geq 0\,:\, \mcl B_t\not\in\BB H\backslash \eta'([0,\wt T]) \}, \qquad
\wh\tau = \inf\{t\geq 0\,:\, \op{Im} \mcl B_t=2 \}.
\end{split}
\eqen
By conformal invariance of Brownian motion and the explicit expression for the Poisson kernel of $\BB H$, see \cite[Exercise 2.23]{lawler-book}, we have
\eqb
\begin{split}
\lambda(g_t(S)) 
&= \lim_{y\rta\infty} \pi y \P^{iy}[\mcl B_{\ol\tau}\in S\,|\,\mcl F_{\wt T}]\\
&\geq \lim_{y\rta\infty} \pi y \P^{iy}[\mcl B_{\wh\tau}\in I_{\eta'(\wt T)}\,|\,\mcl F_{\wt T}] \times \inf_{z\in I_{\eta'(\wt T)}} \P^z[\mcl B_{\ol\tau}\in S\,|\,\mcl F_{\wt T}].
\end{split}
\label{eq:eucl6}
\eqe
By using the explicit formula for the Poisson kernel of $\BB H$ it holds a.s.\ that
\eqb
\lim_{y\rta\infty} \pi y \P^{iy}[\mcl B_{\wh\tau}\in I_{\eta'(\wt T)}\,|\,\mcl F_{\wt T}]=2. 
\label{eq:eucl7}
\eqe
For each $z\in\BB H$ let $K'_z$ be the union of the line segments $[z-i-1,z-1]$ and $[z-1,z]$, let $S'_z$ be the subset of the boundary of $\BB H\backslash K'_z$ corresponding to the lower part of $[z-1,z]$ (viewing the boundary of $\BB H\backslash K'_z$ as a collection of prime ends), and let $\tau':=\inf\{t\geq 0\,:\, \mcl B_t\not\in\BB H\backslash K'_{\eta'(\wt T)}\}$. If $\op{Im} \mcl B_0\geq 2$ it holds by a geometric argument that $\{\mcl B_{\tau'}\in S'_{\eta'(\wt T)}\}\subset \{\mcl B_{\ol\tau}\in S\}$. Therefore 
\eqbn
\begin{split}
	\inf_{z\in I_{\eta'(\wt T)}} \P^z[\mcl B_{\ol\tau}\in S\,|\,\mcl F_{\wt T}]
	\geq \inf_{z\in I_{\eta'(\wt T)}} \P^z[\mcl B_{\tau'}\in S'_{\eta'(\wt T)}\,|\,\mcl F_{\wt T}]
	= \inf_{z\in I_i} \P^z[\mcl B_{\tau'}\in S'_i]\succeq 1.
\end{split}
\eqen
This estimate combined with~\eqref{eq:eucl6} and~\eqref{eq:eucl7} implies the assertion of the lemma.
\end{proof}

\begin{proof}[Proof of upper bound in~\eqref{eq:eucl1} for $T=1$]
It is sufficient to prove that given any $u>0$
\begin{equation}
	\P[E^1_{z_L,z_R}]\preceq (z_Rz_L)^{-u}z_R^{\rho/\kappa'} z_L^{\rho/\kappa'} (z_R+z_L)^{\rho^2/(2\kappa')}.
	\label{eq:eucl11}
\end{equation}
Recall the definition~\eqref{eq:eucl4} of $\sigma_u$. The process $(M_{\sigma_u\wedge t})_{t\geq 0}$ is a bounded martingale, hence the optional stopping theorem implies $\E[M_{\sigma_u}]=M_0$. This implies further that
\eqb
\begin{split}
	\P[M_{\sigma_u}=(z_Rz_L)^{u}]
	&= (z_Rz_L)^{-u}\E[M_{\sigma_{u}}]\\
	&= (z_Rz_L)^{-u}\E[M_0]= (z_Rz_L)^{-u}z_R^{\rho/\kappa'} z_L^{\rho/\kappa'} (z_R+z_L)^{\rho^2/(2\kappa')}.
\end{split}
\label{eq:eucl10}
\eqe
We claim that for each $u>0$ there exists some sufficiently small $s>0$ only depending on $u$ such that $\{ \sigma_u\ge 1\}\subset F_{(z_Rz_L)^s}$ for sufficiently small $z_L,z_R$, with the latter event defined as in Lemma~\ref{lem:eucl1} with $\ep = (z_Rz_L)^s$. Define the stopping time $\wt T_s$ by $\wt T_s:=\inf\{t\geq 0\,:\,\op{Im}(\eta'(t))\geq (z_Rz_L)^s \}$. If $F_{(z_Rz_L)^s}$ does not occur and $M_1\neq 0$, Lemma~\ref{lem:eucl2} implies that $z_{\wt T_s}^L-W_{\wt T_s}>c(z_Rz_L)^s$ and $W_{\wt T_s}-z_{\wt T_s}^R>c(z_Rz_L)^s$, where $c$ is the constant in the statement of the lemma. Hence $M_{\wt T_s}>(z_Rz_L)^u$ for sufficiently small $s,z_L,z_R$, so $\sigma_u<1$ and the claim follows. We have
\eqbn 
E^1_{z_L,z_R}\cap\{M_{\sigma_u}=0 \}\subset \{ \sigma_u\ge 1\}\subset F_{(z_Rz_L)^s} .
\eqen
By Lemma~\ref{lem:eucl1} and~\eqref{eq:eucl10} we have
\eqbn
\begin{split}
\P[E^1_{z_L,z_R}] &\leq \P[F_{(z_Rz_L)^s}] + \P[E^1_{z_L,z_R};M_{\sigma_u}=(z_Rz_L)^{u}]\\
&\preceq z_R^{\rho/\kappa'+o_{u}(1)} z_L^{\rho/\kappa'+o_{u}(1)} (z_R+z_L)^{\rho^2/(2\kappa')}.
\end{split}
\eqen
The result follows since $u > 0$ was arbitrary.
\end{proof}

\subsection{Euclidean lower bound}
\label{sec:lower-bound-Euc}
Recall the definition of $E^1_{z_L,z_R}$ and $(M_t)_{t\geq 0}$ in Proposition~\ref{prop:Euclidean0} and Lemma~\ref{thm:martingale}, respectively. In this section we will prove that $\P[E^1_{z_L,z_R}]$ is bounded below by $\E[M_1]=M_0$ up to $o(1)$ errors in the exponents. In order to prove this estimate we need to show that the contribution to $\E[M_1]$ of large values of $M_1$ is very small.

\begin{proof}[Proof of lower bound in~\eqref{eq:eucl1} for $T=1$]
Since $(M_t)_{t\geq 0}$ is a martingale by Lemma~\ref{lem:eucl3}, we have that
\begin{equation*}
M_0=\E[M_1]=\E[M_1;M_1>(z_Rz_L)^{-u}]+\E[M_1;0<M_1\le(z_Rz_L)^{-u}].
\end{equation*} 
Since $|z_t^q - W_t| \leq |z_t^R - z_t^L|$ for $q \in \{L,R\}$ and all of the exponents in the definition of $M_t$ are positive and sum to $\tfrac{\kappa'}{2}-2$, we have that $M_t\leq |z^R_t-z^L_t|^{\kappa'/2-2}$.  Therefore,
\begin{equation} 
\begin{split}
\E\!\left[M_1; M_1>(z_Rz_L)^{-u}\right]&\le \E\!\left[ |z^R_1-z^L_1|^{\kappa'/2-2};|z^R_{1}-z^L_1|^{\kappa'/2-2}>(z_Rz_L)^{-u}\right]\\
&=o_{z_Rz_L}^\infty(z_Rz_L),
\end{split}
\label{eq1:eucl8}
\end{equation} 
where the last equality follows from large deviation estimates for Bessel processes as in the proof of Lemma~\ref{lem:eucl3}. It follows that $\E[M_1; M_1>(z_Rz_L)^{-u}]<\frac 12 M_0$ if either $z_R$ or $z_L$ is sufficiently small, and therefore
\begin{equation*}
M_0\preceq \E[M_1;0<M_1\le(z_Rz_L)^{-u}]
\le(z_Rz_L)^{-u}\P(M_1>0),
\end{equation*}
which implies that
\eqbn
\P[E^1_{z_L,z_R}]=
\P[M_1>0]\succeq z_R^{\rho/\kappa'+o_{z_R}(1)} z_L^{\rho/\kappa'+o_{z_L}(1)} (z_R+z_L)^{\rho^2/(2\kappa')}. \qedhere
\eqen
\end{proof}

\subsection{Proof of Proposition~\ref{prop:Euclidean0}}
By the scaling property of SLE it is sufficient to prove the estimates~\eqref{eq:eucl1} and~\eqref{eq:eucl2} for $T=1$ and $r=1$, respectively. Combining the results of Sections~\ref{sec:upper-bound-Euc} and~\ref{sec:lower-bound-Euc} we have proved~\eqref{eq:eucl1} for $T=1$. To prove the estimate~\eqref{eq:eucl2} and hence complete the proof of Proposition~\ref{prop:Euclidean0}, it suffices to show that $T_1$ is of order 1 with high probability. 

\begin{proof}[Proof of Proposition~\ref{prop:Euclidean0}]
The lower bound of~\eqref{eq:eucl2} is immediate from~\eqref{eq:eucl1}, since the half-plane capacity of $\eta'$ stopped upon hitting $\partial \BB D \cap \BB H$ is bounded above by the half-plane capacity of $\BB D \cap \BB H$, which implies that we a.s.\ have $T_1\preceq 1$. To complete the proof of the proposition we need to prove the upper bound of~\eqref{eq:eucl2}. Let $\lambda$ denote Lebesgue measure on $\BB R$. By \cite[Equation (3.14)]{lawler-book} and the surrounding text we have
\eqbn
\lambda(g_{T_1}(\eta'([0,T_1]))) \geq c > 0,
\eqen
where the constant $c$ is universal. Conditioned on $E^{T_1}_{z_L,z_R}$ we have
\eqbn
(z^R_{T_1}-W_{T_1})+ (W_{T_1}-z^L_{T_1}) \geq \lambda(g_{T_1}(\eta'([0,T_1]))),
\eqen
hence at least one of the following inequalities holds on $E^{T_1}_{z_L,z_R}$: $z^R_{T_1}-W_{T_1}\geq c/2$ or $W_{T_1}-z^L_{T_1}\geq c/2$. By large deviation estimates for Bessel processes as in the proof of Lemma~\ref{lem:eucl3} we have that for any $u>0$
\eqb \label{eq:eucl-error}
\P\!\left[E^{T_1}_{z_L,z_R}; T_1<(z_Lz_R)^u \right] \preceq
\sum_{q \in \{L,R\}} \P\!\left[\sup_{t\in[0, (z_Lz_R)^u]} \left| z^q_{t}-W_{t} \right| \geq c/2\right]
= o_{z_Lz_R}^\infty(z_Lz_R).
\eqe 
We conclude the proof of the proposition by observing that
\eqbn
	\P[E^{T_1}_{z_L,z_R}] \leq 
	\P[E^{(z_Lz_R)^u }_{z_L,z_R}
	]
	+ \P[E^{T_1}_{z_L,z_R}; T_1<(z_Lz_R)^u ],
\eqen
which by~\eqref{eq:eucl11} and~\eqref{eq:eucl-error} implies that
\eqbn
	\P[E^{T_1}_{z_L,z_R}] =(z_L+z_R)^{\rho^2/(2\kappa')}
	z_L^{\rho/\kappa'+o_{z_L}(1)}
	z_R^{\rho/\kappa'+o_{z_R}(1)}. \qedhere
\eqen
\end{proof}

\section{The quantum exponent}
\label{sec:KPZ}

In this section, we will calculate the probability of a certain event associated with the $\gamma$-quantum boundary measure of a $\frac32\gamma$-quantum wedge.  Throughout this section we will make use of the convention introduced in Section~\ref{sec:surface}, namely we fix $\kappa' > 8$ and $\gamma = 4/\sqrt{\kappa'} \in (0,\sqrt 2)$ and do not make dependence on $\kappa',\gamma$ explicit.  The main result of this section is the following proposition. 

\begin{prop}
\label{prop:quantum-exponent}
Let $\frk h$ be the circle average embedding of a $\frac32\gamma$-quantum wedge in $(\BB H , 0, \infty)$, as in Definition~\ref{def:wedge-free}, let $\nu_{\hwedge}$ be the $\gamma$-quantum boundary measure induced by $\hwedge$, and let $\eta'$ be a chordal $\op{SLE}_{\kappa'}$ in $\BB H$ from $0$ to $\infty$ independent of $\hwedge$. For $r > 0$, let
\eqb \label{eq:hit-time}
T_r := \inf\left\{t > 0 \,:\, |\eta'(t)| = r\right\}  
\eqe  
and for $\delta >0$ let
\eqb \label{eq:hit-time-event}
E_\delta^{T_r} := \left\{ \nu_{\hwedge}\left(\eta'([0,T_r])  \cap \BB R_- \right) \leq \delta \: \op{and} \: \nu_{\hwedge}\left(\eta'([0,T_r]) \cap \BB R_+ \right) \leq \delta  \right\} . 
\eqe 
For each fixed $r\in (0,1]$, we have
\eqbn
\P\!\left[E_\delta^{T_r} 
\right] = \delta^{4/\gamma^2 + o_\delta(1) } .
\eqen
\end{prop}

\subsection{Moment estimates for the quantum boundary measure}
\label{sec:moments}

In this subsection we will state some estimates for the moments of a certain quantity associated with the quantum boundary measure induced by a free-boundary GFF on~$\BB H$, which will be proven in the next two subsections.
Let~$Q$ be as in~\eqref{eq:Q-def} and fix $\alpha \in [0,Q)$. For the proof of Proposition~\ref{prop:quantum-exponent} we only need the case where $\alpha = \frac32\gamma$ (note $\alpha < Q$ for $\gamma \in (0,\sqrt 2)$), but it is no more difficult to treat the general case. Also let
\eqb \label{eq:a-def}
a := Q - \alpha .
\eqe  
Let $h^F$ be a free boundary GFF on $\BB H$, normalized so that its semicircle average over $\bdy B_1(0) \cap \BB H$ is 0. Let $h := h^F - \alpha \log |\cdot|$, so that $h$ is an unscaled $\alpha$-quantum wedge as defined in \cite[Section~1.4]{wedges}.  Let $\nu_h$ be the $\gamma$-quantum boundary measure induced by $h$. 

Fix $r \in (0,1]$. For $\delta > 0$, let $x_{\delta,L} $ and $x_{\delta,R}$ be the non-negative random variables such that $\nu_h([-x_{\delta,L} , 0]) = \nu_h([0,x_{\delta,R}]) = \delta$. Let $\ol x_{\delta,L} =x_{\delta,L} \wedge r$ and $\ol x_{\delta,R} =x_{\delta,R} \wedge r$. In this subsection we will compute the joint moments of $\ol{x}_{\delta,L}$ and $\ol{x}_{\delta,R}$. This calculation, together with the estimate~\eqref{eq:eucl2} of Section~\ref{sec:Euclidean}, will be used to compute the probability in Proposition~\ref{prop:quantum-exponent}. 

\begin{proposition}\label{prop:free-boundary-length}
Let $\ol x_{\delta,L}$ and $\ol x_{\delta,R}$ for $\delta>0$ be as above. For $\lambda_1 , \lambda_2 > 0$ we have
\eqb \label{eq:exponent-quantum-length}
\lim_{\delta \rta 0} \frac{ \log \E\!\left[\ol x_{\delta,L}^{\lambda_1} \ol x_{\delta,R}^{\lambda_2} \right]}{\log \delta^{-1} }
 = \frac{a- \sqrt{a^2 + 4(\lambda_1 + \lambda_2)}  }{\gamma} ,
\eqe  
with $a$ as in~\eqref{eq:a-def}.
\end{proposition}

We will deduce Proposition~\ref{prop:free-boundary-length} from two similar propositions which concern moments of only a single random variable (rather than joint moments) and imply the upper and lower bounds in Proposition~\ref{prop:free-boundary-length}, respectively. 

\begin{proposition}
\label{prop:shifted-KPZ}
Suppose we are in the setting of Proposition~\ref{prop:free-boundary-length}. For each $ \lambda>0$,
\begin{align}
\label{eq:shifted-KPZ}
\lim_{\delta\rta 0 } \frac{\log \E\!\left[\ol x_{\delta,L}^{ \lambda  }\right]}{\log\delta^{-1}}
= \frac{a -  \sqrt{a^2+4\lambda} }{\gamma},
\end{align} 
with $a$ as in~\eqref{eq:a-def}.
\end{proposition}
 
\begin{proposition}
\label{prop:boundary-KPZ}
Let $h$ be an unscaled $\alpha$-quantum wedge as above. For $\delta>0$, let $x_\delta$ be such that $\nu_h([-x_\delta , x_\delta])=\delta$. Also fix $r>0$ and let $\ol x_\delta := x_\delta \wedge r$. Then  
\begin{align*}
\lim_{\delta\rta  0 } \frac{\log \E\!\left[\ol x_{\delta  }^\lambda \right]}{\log\delta^{-1}  }
= \frac{a -  \sqrt{a^2+4\lambda} }{\gamma},
\end{align*}  
with $a$ as in~\eqref{eq:a-def}.
\end{proposition} 

The following lemma tells us that in order to prove Propositions~\ref{prop:free-boundary-length},~\ref{prop:shifted-KPZ}, and~\ref{prop:boundary-KPZ}, we need only prove the upper bound for the limit in Proposition~\ref{prop:shifted-KPZ} and a lower bound for the limit in Proposition~\ref{prop:boundary-KPZ}.

\begin{lemma} \label{lem:kpz-equiv}
Let $\ol x_{\delta, L} , \ol x_{\delta,R}$ be defined as in the beginning of this subsection and let $\ol x_\delta$ be as in Proposition~\ref{prop:boundary-KPZ} (with the same choice of $r$). For each $\lambda_1 , \lambda_2 > 0$, we have
\eqb \label{eq:kpz-equiv}
 \lim_{\delta\rta 0} \frac{\log \E\!\left[\ol x_\delta^{\lambda_1+\lambda_2}\right]}{\log\delta^{-1} }  \leq  \lim_{\delta \rta 0} \frac{ \log \E\!\left[\ol x_{\delta, L}^{\lambda_1} \ol x_{\delta, R}^{\lambda_2} \right]}{\log \delta^{-1}} \leq \lim_{\delta \rta 0 } \frac{\log \E\!\left[\ol x_{\delta, L}^{ \lambda_1 + \lambda_2  }\right]}{\log\delta^{-1}} .
\eqe 
\end{lemma}
\begin{proof}
By definition, we have $\ol x_{\delta } \leq \ol x_{\delta, L} \wedge \ol x_{\delta, R}$, which gives the first inequality in~\eqref{eq:kpz-equiv}.   
The second inequality follows from
\alb
\E\!\left[\ol x_{\delta, L}^{\lambda_1} \ol x_{\delta, R}^{\lambda_2}\right] 
&\leq \E\!\left[(\ol x_{\delta, L}+ \ol x_{\delta, R} )^{\lambda_1+\lambda_2}\right]\\
&\leq 2^{\lambda_1+\lambda_2}\E\!\left[\ol x_{\delta, L}^{\lambda_1+\lambda_2}+\ol x_{\delta, R}^{\lambda_1+\lambda_2}\right] \\
&= 2^{\lambda_1+\lambda_2+1}\E\!\left[\ol x_{\delta, L}^{\lambda_1+\lambda_2}\right]. \qedhere
\ale
\end{proof}

The proofs of the lower bound in Proposition~\ref{prop:shifted-KPZ} and the upper bound in Proposition~\ref{prop:boundary-KPZ} will be completed in the next two subsections. Both proofs use arguments similar to those found in~\cite[Section~4]{shef-kpz}. In particular, both estimates are established by first proving the semicircle average version of the estimate and then showing that the exponential of $\gamma$ times the semicircle average is in some sense a good approximation for the quantum measure.

\subsection{Circle average KPZ and tail estimates}
\label{sec:circle-avg}

In this subsection we will establish several lemmas which are similar to various results in~\cite[Section~4]{shef-kpz} and which are needed for the proofs of the results in Section~\ref{sec:moments}. Throughout this subsection and the next, we assume we are in the setting of Section~\ref{sec:moments}, and we use the notation introduced there plus  the following additional notation. For $\ep > 0$, let $h_\ep(z)$ be the semicircle average of $h$ about $\partial B_\ep(z) \cap \BB H$.
For $t\in \BB R$, let 
\eqb \label{eq:bm-def}
 V_t := -h_{e^{-t}} (0)+Q t, 
\eqe 
with $Q$ as in~\eqref{eq:Q-def}.
As explained in~\cite[Section~6.1]{shef-kpz}, $V_t$ is distributed as $\mcl B_{2t}+at$ where $\mcl B $ is a standard linear two-sided Brownian motion and $a$ is as in~\eqref{eq:a-def} (here we recall that $h$ has a $-\alpha$-log singularity at 0).  
 
Let 
\eqb \label{eq:hitting-time}
 A_\delta :=  \frac{2}{\gamma} \log \delta^{-1}  \quad \op{and} \quad \tau_\delta :=\inf\{t \geq 0: V_t= A_\delta \}.
\eqe  
As we will see, $\exp(-\tau_\delta)$ is a good estimator of $\ol x_{\delta, L}$. The semicircle average version of Proposition~\ref{prop:shifted-KPZ} is the following simple fact regarding Brownian motion.

\begin{lemma}\label{lem:circle-KPZ}
For $\lambda>0$ we have
\begin{equation}\label{eq:eps1=2}
\lim_{\delta\rta  0 } \frac{\log \E\!\left[e^{-\lambda \tau_\delta} \right]}{\log\delta^{-1} }=
\frac{a -  \sqrt{a^2+4\lambda} }{\gamma} .
\end{equation} 
\end{lemma}
\begin{proof}
Write $V_t = \mcl B_{2t}+at$, with $\mcl B$ a standard linear Brownian motion. Let 
\eqbn
\beta := \frac{ \sqrt{a^2+4\lambda}-a}{2}
\eqen 
so that $\beta^2+a\beta = \lambda$.   
We observe that $t\mapsto \exp\left(\beta \mcl B_{2t} - \beta^2 t\right)$ is a non-negative martingale. Furthermore, using that $\alpha < Q$ so that $a  > 0$, for $t\leq \tau_\delta$ we have $\mcl B_{2t} \leq A_\delta$.  In particular, $t \mapsto \exp\left(\beta \mcl B_{2 \tau_\delta \wedge t} - \beta^2 \tau_\delta \wedge t\right)$ is bounded and $\BB P[\tau_\delta < \infty] = 1$.
By the optional stopping theorem,
\[
\E\!\left[ \exp\left( \beta \mcl B_{2\tau_\delta} -\beta^2 \tau_\delta \right) \right] =1.
\]
Since $\mcl B_{2\tau_\delta} = A_\delta -  a\tau_\delta$, we have
\eqbn
\E\!\left[e^{-\lambda \tau_\delta}\right]
=  e^{-\beta A_\delta} = \delta^{2\beta/\gamma} = \delta^{\frac{ \sqrt{a^2+4\lambda}-a}{\gamma}}
\eqen
which implies the statement of the lemma.
\end{proof}

To deduce Proposition~\ref{prop:shifted-KPZ} from Lemma~\ref{lem:circle-KPZ}, we first need a lower bound for the $\gamma$-quantum boundary length of an interval. The needed estimate can be deduced in a similar manner to~\cite[Lemma~4.5]{shef-kpz}, but for brevity we give an alternative argument based on the theory of Gaussian multiplicative chaos~\cite{kahane,rhodes-vargas-review}.

\begin{lemma}
\label{lem:free-gff-moments}
Let $h^F$ be as in Section~\ref{sec:moments} and let $\nu_{h^F}$ be its associated $\gamma$-quantum boundary measure. Let $I \subset \partial \BB H$ be a bounded open interval. Then $\nu_{h^F}(I)$ has finite moments of all negative orders.
\end{lemma}
\begin{proof}
This follows from general Gaussian multiplicative chaos theory applied to $\nu_{h^F}$. See, e.g.\ \cite[Theorems 2.11 and 2.12]{rhodes-vargas-review}. See also \cite[Section~4.4]{qle} for an approximation scheme for $\nu_{h^F}$ to which Gaussian multiplicative chaos theory applies (the approximation scheme is stated in the context of the unit disk $\BB D$, but a similar formula works for the upper half plane).  
\end{proof}
  
\begin{lemma}
\label{lem:nu_h-tail}
Let $\tau$ be a stopping time for the filtration $\mcl F_t = \sigma(V_s \,:\, s \in (-\infty , t] )$. Also fix $u >0$. For each $\delta \in (0,1)$, we have
\eqb \label{nu_h tails disk eqn}
\P\!\left[ \nu_h([0,e^{-\tau}])     <    \delta^u \exp\left(  - \frac{\gamma}{2} V_\tau   \right)   \,|\,  \mcl F_\tau  \right] = o_\delta^\infty(\delta)  
\eqe
at a deterministic rate which does not depend on $\tau$. 
\end{lemma}
\begin{proof} 
Fix a stopping time $\tau$ as in the statement of the lemma.  The restriction of $h$ to $B_{e^{-\tau}}(0) \cap \BB H$ is determined by the orthogonal projection of $h$ onto the set of functions with mean zero on all semicircles centered at 0 together with the normalized semicircle averages $V_t$ for $t \geq \tau$ (c.f.~\cite[Lemma~4.1]{wedges}). Since $t\mapsto V_t$ has the law of a two-sided drifted Brownian motion normalized to vanish at 0, it follows from the strong Markov property of Brownian motion that the conditional law given $\mcl F_\tau$ of the restriction of $h- h_{e^{-\tau}}(0)$ to $B_{e^{-\tau}}(0) \cap \BB H$ is that of a free boundary GFF restricted to $B_{e^{-\tau}}(0) \cap \BB H$ and normalized so that its semicircle average vanishes on $\bdy B_{e^{-\tau}}(0) \cap \BB H$. It follows from the construction of $\nu_h$ via semicircle averages (see~\cite[Section~6]{shef-kpz}) that $e^{-\frac{\gamma}{2} h_{e^{-\tau}}(0)} \nu_h([0,e^{-\tau}])$ is determined by the restriction of $h- h_{e^{-\tau}}(0)$ to $B_{e^{-\tau}}(0) \cap \BB H$. 

Let $\phi (z) := e^{-\tau} z$. 
Let $\wt h = h\circ \phi      + Q \log e^{-\tau} $, with $Q$ as in~\eqref{eq:Q-def}. By the boundary analogue of~\cite[Proposition~2.1]{shef-kpz}, we have $\nu_h([0, e^{-\tau}]) = \nu_{\wt h}([0,1])$.  
Let $\wt h_*$ be the restriction to $B_1(0) \cap \BB H$ of the field $h\circ \phi - h_{e^{-\tau}}(0)$. By conformal invariance of the free boundary GFF and the discussion above, it follows that the conditional law given $\mcl F_\tau$ of the restriction of $h  \circ \phi  $ to $B_1(0) \cap \BB H$ is the same as the law of $h |_{B_1(0) \cap \BB H}$, modulo a global additive constant. The semicircle average of $h\circ \phi$ over $\bdy B_1(0)\cap\BB H$ is given by $h_{e^{-\tau}}(0)$. It therefore follows that the conditional law of $\wt h_* $ given $\mcl F_\tau$ is the same as the law of $ h|_{B_1(0) \cap \BB H}$. 
 
By the definition of the $\gamma$-quantum boundary measure we have
\eqb \label{eq:coord-new-gff}
\nu_{\wt h}([0,1])
= \exp\left(  \frac{\gamma}{2} h_{e^{-\tau}}(0) - \frac{\gamma}{2} Q \tau \right) \nu_{\wt h_*}([0,1])
= \exp\left(   - \frac{\gamma}{2} V_\tau   \right) \nu_{\wt h_*}([0,1]).
\eqe
By Lemma~\ref{lem:free-gff-moments}, the conditional law given $\mcl F_\tau$ of $\nu_{\wt h_*}([0,1])$ has moments of all negative orders, so by Chebyshev's inequality, for each $\delta>0$ we have that $\P\!\left[\nu_{\wt h_*}([0,1]) \leq \delta^u \,|\, \mcl F_\tau \right]  $ decays faster than any power of $\delta$. We thus obtain the statement of the lemma.
\end{proof}

\subsection{Proof of the moment estimates}
\label{sec:moment-proof}

In this subsection we will prove the upper bound in Proposition~\ref{prop:shifted-KPZ} and the lower bound in Proposition~\ref{prop:boundary-KPZ}, thereby completing the proof of the propositions in Section~\ref{sec:moment-proof}. Our first proof is similar to the argument given in~\cite[Section~4.4]{shef-kpz}. 

\begin{proof}[Proof of Proposition~\ref{prop:shifted-KPZ}, upper bound]
Fix $s \in (0,1)$. For $\delta>0$, let $\tau_{\delta^s}$ be as in~\eqref{eq:hitting-time} with $\delta^s$ in place of $\delta$, so that with $A_\delta$ as in~\eqref{eq:hitting-time} we have
\eqbn
\tau_{\delta^s} = \inf\{t \geq 0 \,:\, V_t = s A_\delta \} .
\eqen
Let $\wh x_{\delta,s} :=\exp(-\tau_{\delta^s})$. For $\lambda > 0$, we have
\begin{align} \label{eq:exponent-split}
 \E\!\left[\ol x_{\delta, L}^{ \lambda  }\right] \leq \E\!\left[\wh x_{\delta,s}^{ \lambda  }\right]   +   \E\!\left[ \ol x_{\delta, L}^{\lambda}  \,; \,  \wh x_{\delta,s}  \leq \ol x_{\delta, L}  \right].
\end{align} 
By Lemma~\ref{lem:circle-KPZ} (applied with $\delta^s$ in place of $\delta$) we have
\eqb \label{eq:exponent-calc}
\lim_{\delta \rta 0} \frac{\log \E\!\left[\wh x_{\delta,s}^{ \lambda  }\right]    }{\log\delta^{-1} } = s \frac{a   -  \sqrt{a^2+4\lambda} }{\gamma} .
\eqe 
On the event $\{  \wh x_{\delta,s}   \leq \ol x_{\delta, L}    \}$ we have  
\eqbn
\nu_h([0, \wh x_{\delta,s}])  \leq  \delta =    \delta^{1-s}  \exp\left( -   \frac{\gamma}{2} V_{  \tau_{\delta^s} } \right)  .
\eqen 
By Lemma~\ref{lem:nu_h-tail}, $\P[ \wh x_{\delta,s}  \leq \ol x_{\delta, L}] = o_\delta^\infty(\delta)$. By definition, we have $\ol x_{\delta, L} \leq r$, so $\E\!\left[ \ol x_{\delta, L}^{\lambda} ; \wh x_{\delta,s}  \leq \ol x_{\delta, L} \right]$ decays faster than any power of $\delta$. Since $s$ is arbitrary, the desired upper bound now follows from~\eqref{eq:exponent-split} and~\eqref{eq:exponent-calc}. 
\end{proof}

Finally we prove the lower bound in Proposition~\ref{prop:boundary-KPZ}.

\begin{proof}[Proof of Proposition~\ref{prop:boundary-KPZ}, lower bound]
For $\delta>0$ let $\tau_\delta$ be as in~\eqref{eq:hitting-time} and let $\wh x_\delta := e^{-\tau_\delta}$. We note that $\tau_\delta \geq 0$, so $\wh x_\delta\leq 1$. Also let $\mcl F_{\tau_\delta} := \sigma(V_t \,:\, t \leq \tau_\delta)$. We claim there exists a constant $c > 0$ (independent of $\delta$) such that
\eqb \label{eq:smaller-ball}
\P\!\left[ \wh x_\delta \leq \ol x_\delta \,|\, \mcl F_{\tau_\delta} \right] = \P\!\left[ \nu_h([-\wh x_\delta , \wh x_\delta]) \leq \delta      \,|\, \mcl F_{\tau_\delta} \right] \geq c \quad\text{a.s.}
\eqe
Assuming that~\eqref{eq:smaller-ball} holds, we get that for $\lambda>0$,
\eqbn
\E\!\left[  \ol x_\delta^\lambda \right] 
\geq \E\!\left[  \wh x_\delta^\lambda \P\!\left[ \wh x_\delta \leq \ol x_\delta \,|\, \mcl F_{\tau_\delta} \right] \right] 
\geq c \E\!\left[ \wh x_\delta^\lambda \right] ,
\eqen 
which implies the desired lower bound. 

It remains only to prove~\eqref{eq:smaller-ball}. 
To this end, we define $\phi$, $\wt h$, and $\wt h_*$ as in the proof of Lemma~\ref{lem:nu_h-tail} with $\tau = \tau_\delta$, so that the conditional law of $\wt h_*$ given $\mcl F_{\tau_\delta}$ is the same as the law of $h|_{B_1(0) \cap \BB H}$; and (as in~\eqref{eq:coord-new-gff}) we have 
\eqbn
\nu_h([-\wh x_\delta , \wh x_\delta]) = \exp\left(   - \frac{\gamma}{2} V_{\tau_\delta}  \right) \nu_{\wt h_*}([-1,1]) = \delta \nu_{\wt h_*} ([-1,1]).
\eqen
It is easy to see that $\P\!\left[\nu_{\wt{h}_*}([-1,1])  \leq 1 \right] > 0$, and~\eqref{eq:smaller-ball} follows. 
\end{proof}

\subsection{Proof of Proposition~\ref{prop:quantum-exponent}}
\label{sec:wedge-free}

Since $\hwedge$ is the circle average embedding of a quantum wedge and $h$ is normalized so that its semicircle average over $\bdy B_1(0) \cap\BB H$ vanishes, Remark~\ref{rmk:wedge-free} implies that we can couple $h$ and $\hwedge$ such that $h \equiv \hwedge$ on $\BB D\cap\BB H$. For $\delta>0$, let $\xwedge_{\delta,L}$ and $\xwedge_{\delta,R}$ be chosen so that $\nu_{\hwedge}([-\xwedge_{\delta,L} , 0])  = \nu_{\hwedge}([0,\xwedge_{\delta,R}])  = \delta$ (as in Remark~\ref{rmk:E-equiv}). By our choice of coupling we have $ \xwedge_{\delta,L} \wedge r = \ol{x}_{\delta,L} $ and $\ol{\xwedge}_{\delta,R} \wedge r = \ol{x}_{\delta,R}$. 

Assume that the SLE curve $\eta'$ is sampled independently from $h$ and $\hwedge$. Then $E_\delta^{T_r}$ is the event that $\eta'$ reaches $\bdy B_r(0)$ before hitting either $-\ol{\xwedge}_{\delta,L}$ or $\ol{\xwedge}_{\delta,R}$ (in particular, $E_\delta^{T_r}$ occurs a.s.\ if $\ol{\xwedge}_{\delta,L}  > r$ and $\ol{\xwedge}_{\delta,R}  > r$). By Proposition~\ref{prop:Euclidean}, for each $u > 0$ we have 
\begin{align}
\label{eq:E-prob-split} 
\ol{x}_{\delta,L}^{\rho/\kappa'+u} \ol{x}_{\delta,R}^{\rho/\kappa'+u}     ( \ol{x}_{\delta,L}+\ol{x}_{\delta,R} )^{\rho^2/(2\kappa')} 
\preceq \P\!\left[ E_\delta^{T_r} \,|\, \hwedge \right]   \preceq \ol{x}_{\delta,L}^{\rho/\kappa'-u} \ol{x}_{\delta,R}^{\rho/\kappa'-u}     ( \ol{x}_{\delta,L}+\ol{x}_{\delta,R} )^{\rho^2/(2\kappa')}  
\end{align} 
with $\rho = \kappa'-4$, as in Section~\ref{sec:Euclidean}, and the implicit constants deterministic and independent of $\delta$ (but possibly depending on $r$ and $u$). 
From the inequality
\eqbn
\frac{1}{2}(\ep_1^p+\ep_2^p) \leq (\ep_1+\ep_2)^p\leq 2^p(\ep_1^p+\ep_2^p)\quad \text{for all} \quad p , \ep_1 , \ep_2 >0 
\eqen
and symmetry between $\ol{x}_{\delta,L}$ and $\ol{x}_{\delta,R}$, 
we infer that the expectations of the left and right sides of~\eqref{eq:E-prob-split} are bounded above and below by constants (depending only on $\kappa'$) times
\begin{align*}
 \E\!\left [\ol{x}_{\delta,L}^{\rho/\kappa' + o_u(1)} \ol{x}_{\delta,R}^{(\rho^2+2\rho)/(2\kappa')+ o_u(1)}  \right]   
\end{align*}
where here the $o_u(1)$ is deterministic and independent of $\delta$. 
By Proposition~\ref{prop:free-boundary-length} applied with $\alpha = \frac32\gamma$, this latter quantity is of order $\delta^{4/\gamma^2 + o_u(1)}$. Since $u$ is arbitrary, we obtain the proposition. \qed

\section{Conclusion of the proof}
\label{sec:embedding}

 In this section we will deduce~\eqref{eq:exponent-compare} from Proposition~\ref{prop:quantum-exponent} and thereby complete the proof of Theorem~\ref{thm:main}. Suppose we are in the setting of Section~\ref{subsec:main}. Define the $\frac32\gamma$-quantum wedge $\mcl W = (\BB H , \hwedge , 0 , \infty)$ and the $\op{SLE}_{\kappa'}$ curve $\wt\eta'$ as in Remark~\ref{rmk:E-equiv}. For $t > 0$ and $\delta>0$, let $\wt E_\delta^t$ be as in~\eqref{eq:def-E}. Equivalently, by Remark~\ref{rmk:E-equiv}, 
\eqb
\label{eq:hit-event}
\wt E_\delta^t := \left\{ \nu_{ \hwedge}\!\left(\wt\eta'([0,t]) \cap \BB R_- \right) \leq \delta \: \op{and} \: \nu_{\hwedge}\!\left(\wt\eta'([0,t]) \cap \BB R_+ \right) \leq \delta  \right\}  .
\eqe
For $r > 0$ let $T_r$ be as in~\eqref{eq:hit-time} and let $E^{T_r}_\delta$ be as in~\eqref{eq:hit-time-event}.  

Roughly speaking, we will show that $\P[  \wt E^1_\delta]$ is a good approximation of $\P[E^{T_1}_\delta]$. Combining this with Proposition~\ref{prop:quantum-exponent} will complete the proof of~\eqref{eq:exponent-compare}. Showing that $\P[E^{T_1}_\delta]$ is less than or equal to $\P[\wt E_\delta^1]$ (up to $o(1)$ error in the exponent)  is relatively simple. One just needs to notice that when $\eta'$ exits the Euclidean unit ball, with overwhelmingly high probability it will contain a Euclidean ball of radius $\delta^{o_\delta(1)}$, and will therefore have quantum mass at least $\delta^{o_\delta(1)}$ with overwhelmingly high probability. Therefore $ \wt E^{t}_\delta$ occurs for some $t\ge \delta^{o_\delta(1)}$. The details are provided in Section~\ref{subsec:Brownian-upper}.

The upper bound  of  $\P[\wt E_\delta^1]$ in terms of $\P[E^{T_1}_\delta]$ is more difficult. One could worry that if  $\wt E_\delta^1$ occurs, then the Euclidean size of $\wt{\eta}'([0,1])$  under the circle average embedding is very small with high probability. This scenario cannot be ruled out directly by using the quantum mass tail estimate because the upper tail only has a power law decay (see \cite[Theorems 2.11 and 2.12]{rhodes-vargas-review}). In Section~\ref{subsec:Brownian-lower}, by exploring the relationship between various embeddings of a $\frac32 \gamma$-quantum wedge, we will show that  conditioned on $\wt E^1_\delta$, there is a uniformly positive probability that $E^{T_r}_\delta$ occurs for some  $\delta$-independent constant $r>0$.

\subsection{Upper bound for $\sigma(\gamma)$}\label{subsec:Brownian-upper}
We first prove an analogue of \cite[Proposition~10.13]{wedges}, which in turn is an analogue of \cite[Lemma~4.5]{shef-kpz}.
\begin{proposition}\label{prop:SLE-hull}
	Fix $\gamma\in (0,\sqrt{2})$ and let $(\bbH, \hwedge, 0 , \infty)$ be a $\frac32\gamma$-quantum wedge under the circle average embedding (Definition~\ref{def:wedge-free}). Let $ \eta'$ be an independent chordal $\SLE_{\kappa'}$ in $\bbH$ from $0$ to $\infty$. Then with $T_1$ as in~\eqref{eq:hit-time}, we have
	\begin{equation}\label{eq:SLE-hull}
	\P[\mu_{\hwedge}(\eta'([0,T_1])) \leq \delta]=o_\delta^\infty(\delta).
	\end{equation}
\end{proposition}
\begin{proof}
	Let $r$ denote $1/2$ times the radius of the largest Euclidean ball contained in $\eta'([0,T_1])$ and let $z$ be the center of this ball. Then it suffices to show that
	\[
	\P[\mu_{h^F}(B_r(z) )\le \delta]=o_\delta^\infty(\delta),
	\]
where $h^F $ has the law of a free boundary GFF with the additive constant fixed so that the average of $h^F$ on $\BB H \cap \partial \D$ is equal to $0$.
	The reason why we can replace $\hwedge$ by $h^F$ is that $\hwedge|_{\BB D \cap\BB H  }$ agrees in law with the restriction of $h^F - \frac32\gamma \log|\cdot|$ to $\BB D\cap\BB H$ and $ B_r(z) \subset\D\cap\bbH$, so  $-\frac32\gamma\log|\cdot|$ is positive on $B_r(z)$.
	As argued in the proof of \cite[Proposition~10.13]{wedges}, the law of the random variable $ r^{-1}$ has an exponential tail at $\infty$ (although the argument there is for a whole plane, the same argument works for $\bbH$).  In particular, for $\delta>0$ we have 
\[	
\BB P\left[r \leq (\log \delta^{-1})^{-2} \right] = o_\delta^\infty(\delta) .
\]
Conditioned on ${\eta}'([0,T_1])$  (which is independent from $h^F$) the regular conditional law of the circle average $h_r^F(z)$ is  that of  a Gaussian with variance at most $-2\log r$ (see \cite[Section~3.1]{shef-kpz}). Here we use the fact that $B_r(z)$ lies at distance at least~$r$ from~$\R$. By the Gaussian tail bound, for each fixed $s\in (0,1)$ we have
\[	
\P\left[e^{\gamma h_r^F(z)} \le \delta^s  \,|\, r \geq (\log \delta^{-1})^{-2} \right] = o_\delta^\infty(\delta) .
\]
On the other hand, by \cite[Lemma~4.5]{shef-kpz} we have that
\[
\P\left[\mu_{h^F}(B_{r}(z) )\le \delta \,|\,  r \geq (\log \delta^{-1})^{-2} ,\,  e^{\gamma h_r^F(z)} \geq \delta^s    \right] = o_\delta^\infty(\delta) .
\]
The proof concludes. 
\end{proof}

 For fixed $s>0$, we have 
\begin{equation}
\label{eq:Brown-control}
\P[E^{T_{1}}_\delta]\le  \P[ E^{\delta^s}_\delta] +\P[T_1\le \delta^s].
\end{equation}
By Proposition~\ref{prop:SLE-hull}, for each $s>0$,
\eqb \label{eq:bad-event-exponent}
\lim_{\delta\to 0}\frac{\log\P[T_1<\delta^s]}{\log\delta^{-1} }= -\infty .
\eqe
By~Lemma~\ref{lem:cone-exponent}, 
\eqb \label{eq:almost-exponent}
\lim_{\delta\rta 0} \frac{\log \P[ E^{\delta^s}_\delta]}{\log\delta^{-1} } = -(1-s/2)\sigma(\gamma) ,
\eqe 
with $\sigma(\gamma)$ as in~\eqref{eq:sigma-def}. By Proposition~\ref{prop:quantum-exponent}, 
\eqb \label{eq:known-exponent}
\lim_{\delta \rta 0} \frac{\log \P[E^{T_{1}}_\delta]}{\log \delta^{-1}} = -\frac{4}{\gamma^2} .
\eqe
Since $s$ is arbitrary, we can combine~\eqref{eq:Brown-control},~\eqref{eq:bad-event-exponent},~\eqref{eq:almost-exponent}, and~\eqref{eq:known-exponent} to obtain $\sigma(\gamma) \leq 4/\gamma^2$, which is the upper bound in~\eqref{eq:exponent-compare}. 

\subsection{Lower bound for $\sigma(\gamma)$}
\label{subsec:Brownian-lower}

\subsubsection{Notation for quantum surfaces}
\label{subsub:notation}
Let $\cW$ be the $\frac32 \gamma$-quantum wedge in Section~\ref{sec:intro}.
In the remainder of this section, we will consider several different parameterizations of $\mcl W$.
Recall that $\cW$ is an equivalence class of $4$-tuples $(D,\hwedge,a,b)$, with $D\subset \BB C$, $\hwedge$ a distribution on $D$, and $a,b\in\bdy D$ where the equivalence relation is defined in terms of transformations on the form~\eqref{eq:lqg-coord}.

We will consider two coordinate systems: $(\bbH, 0,\infty)$ and $(\cS,+\infty,-\infty)$ where $\cS :=\R\times (0,\pi)$ is a horizontal strip. Whenever we switch between the two coordinate systems, we assume that the corresponding objects are related as in~\eqref{eq:lqg-coord} with $f$ the \emph{canonical coordinate transformation} between the two systems 
\eqb \label{eq:S-to-H}
z \mapsto -e^{-z} ,\quad  z\in \cS .
\eqe
Since a wedge only has two marked boundary points, knowing the coordinate system is not sufficient to determine the embedding of the surface, i.e.\  there is one free parameter corresponding to scaling $\BB H$ or horizontally translating $\mcl S$. We will consider several different embeddings of $\mcl W$ into each of $\bbH$ and $\cS$. We slightly abuse notation by using the same symbols for embeddings into $\bbH$ and $\cS$, always keeping in mind that the corresponding fields are related via the map~\eqref{eq:S-to-H}.  Hereafter, we will denote an embedding of $\hwedge$ in a given coordinate system by $\hwedge^{\bullet}$, where $\bullet$ indicates the particular choice of embedding. After we define $\hwedge^{\bullet}$ in one coordinate system, we simultaneously define $\hwedge^{\bullet}$ in the other coordinate system by applying the coordinate change formula~\eqref{eq:lqg-coord} with the mapping~\eqref{eq:S-to-H}. 

In $(\cS,+\infty,-\infty)$, we let $X^{\bullet}_t$ be the average process of $\hwedge^{\bullet}$ along the vertical line segment $\{t\}\times(0,\pi)$, where $\bullet$ is the symbol representing the embedding. Before fixing the embedding of $\cW$ into ${(\cS,+\infty,-\infty)}$,  the average process is defined up to a horizontal translation. Therefore we can fix the embedding of $\cW$ on $(\cS,+\infty,-\infty)$ by  specifying the translation of the average process. We define the \emph{circle average embedding} of $\mcl W$ into $(\cS,+\infty,-\infty)$ by requiring
$\inf\{ t \in\R: X^{\bullet}_t=0  \}=0$ and denote the field (resp.\ average process) on $\cS$ by $\hwedge^C$ (resp.\ $X^C$). Note that the circle average embedding into $\cS$ is the image of the circle average embedding into $\bbH$ (Definition~\ref{def:wedge-free}) under the coordinate change~\eqref{eq:S-to-H}, and in keeping with our convention the latter embedding will also be denoted by $\hwedge^C$ in the remainder of this section. By the definition of a quantum  wedge and by~\eqref{eq:lqg-coord}, under the circle average embedding in $\cS$ there are standard Brownian motions $\mcl B,\wh{\mcl B}$ such that $X_t^C=\mcl B_{2t}-at$ for $t\geq 0$ and $X_t^C=\wh{\mcl B}_{-2t}-at$ for $t<0$, where $a = Q-\frac32\gamma$ and $\wh{\mcl B}$ is conditioned such that $X_t^C\geq 0$ for $t<0$.

We will use the so-called \emph{unit radius embedding} of $\cW$, which we denote by $\hwedge^U$. On $(\bbH,0,\infty)$ it is defined such that $\diam(\wt\eta'([0,1]))=1$.

We will also consider the so-called \emph{smooth centering embedding}, which we denote by~$\hwedge^S$. It is introduced in the context of quantum cones in~\cite[Section~10.4.2]{wedges}. Let $\phi$ be a fixed positive smooth function supported on $[0,1]$ with integral 1. The smooth centering embedding of $\mcl W$ into $(\cS,+\infty,-\infty)$ is such that
\[ \inf \left\{t\in \R: \int_{-\infty}^\infty X_s\phi(s-t)ds \leq 0 \right\} = 0.\]
Since $\lim_{t\to\infty}X_t=-\infty$ and $\lim_{t\to-\infty}X_t=+\infty$, $\hwedge^S$ is well-defined almost surely.

Different embeddings of $\cW$ into $(\bbH,0,\infty)$ (resp.\ $(\cS,+\infty,-\infty)$) differ by a scaling  (resp.\ horizontal translation). We let $\sigma_{\bullet,\diamond}^{\bbH}(\cW)$ 
	(resp.\ $\sigma_{\bullet,\diamond}^{\cS}(\cW)$) be the possibly random constant $c$ such that $\hwedge^{\diamond}(\cdot)=\hwedge^{\bullet}(c\cdot) + Q\log c$ (resp.\ $\hwedge^{\diamond}(\cdot)=\hwedge^{\bullet}(\cdot+c)$). )  
 
We will also have occasion to consider the quantum surface $$\cW_* :=(\wt\eta'([1,\infty)),\hwedge_*, \wt\eta'(1),\wt\eta'(\infty))$$ obtained by restricting $\cW$ to $\wt\eta'([1,\infty))$. By~\cite[Lemma~9.3]{wedges} (see also the proof of \cite[Lemma~9.2]{wedges}),
$\mcl W_*$
is a $\frac32\gamma$-quantum wedge independent of $(Z_t)_{t\in[0,1]}$.
 We will mainly be interested in $\cW_*$ embedded in $(\bbH, 0,\infty)$ in two ways. One is the circle average embedding $\hwedge_*^C$. The other one is defined as follows.
Consider $\mcl W$ embedded in $(\bbH,0,\infty)$ under the unit radius embedding $\hwedge^U$. Let $\Psi:\bbH\setminus\wt\eta'([0,1]) \to \bbH$ be the conformal map such that $\Psi(\wt\eta'(1))=0,\Psi(\infty)=\infty$, and $\lim_{z\rta\infty} \Psi(z)/z = 1$. Then $\hwedge^{\Psi}_* := \hwedge^U\circ \Psi^{-1} +Q\log|(\Psi^{-1})'|$ gives an embedding of $\mcl W_*$ into $(\bbH,0,\infty)$, which we will call the \emph{$\Psi$-embedding} of $\cW_*$.

\subsubsection{Smooth centering embedding and conclusion of the proof}\label{subsub:lower-diam}
In Section~\ref{subsec:wedge-smooth}, we will prove the following proposition, which is a variant of a result proved in \cite[Section~10.4.2]{wedges} for quantum cones. The proof will use similar techniques as the proof in \cite{wedges}.

\begin{proposition}
\label{prop:wedge-smooth}
	Let $\mcl W = (\bbH,\hwedge^S,0,\infty)$ be a $\frac32\gamma$-quantum wedge with the smooth centering embedding and let $\wt\eta'$ be an independent chordal $\SLE_{\kappa'}$ in $\bbH$ from 0 to $\infty$ parameterized by quantum mass with respect to $\hwedge^S$. There are deterministic constants $c,r>0$ and an event $G$ such that the following is true for all $\delta\in (0,\frac12)$.
	\begin{enumerate}[(i)]
		\item $\P[G\,|\, \wt E^1_\delta]\ge c$. \label{item:G-prob}
		\item On $G\cap \wt E^1_\delta$, $\diam(\wt \eta'([0,1]))>r$. \label{item:G-diam}
	\end{enumerate} 
\end{proposition}
\begin{remark}\label{rmk:ratio}
The condition that $\diam(\wt \eta'([0,1]))>r$ in~\eqref{item:G-diam} above can also be written as $\sigma_{S,U}^{\bbH}(\cW)>r$, in the notation of Section~\ref{subsub:notation}.
\end{remark}
 
Before we prove Proposition~\ref{prop:wedge-smooth} in Section~\ref{subsec:wedge-smooth}, we first explain why it almost implies the lower bound for $\sigma(\gamma)$ in~\eqref{eq:exponent-compare}. 
If Proposition~\ref{prop:wedge-smooth} were true with the circle average embedding $\hwedge^C$ in place of $\hwedge^S$, then the corresponding event $G$ would satisfy $G\cap  \wt E^1_\delta \subset E^{T_r}_\delta$. By condition~\eqref{item:G-prob} in Proposition~\ref{prop:wedge-smooth}, $\P[\wt E_\delta^1] \preceq \P[ E_\delta^{T_r}]$.  Combined with Proposition~\ref{prop:quantum-exponent}, this implies $\sigma(\gamma) \geq 4/\gamma^2$. 

The following simple fact bridges the gap between the smooth centering embedding and the circle average embedding.

\begin{lemma}
\label{lem:bump-circle}
	Let $a>0$ and let $\mcl B_t$ be a standard linear Brownian motion starting from~$0$. For $M>0$, let $\tau_M=\inf\{t\ge 0:\mcl B_t-at =-M\}$. Let
	 $F_M$ be the event that 
\[
\int_{0}^\infty (\mcl B_s-as)\phi(s-t)ds\ge 0,\quad \forall t\in [0,\tau_M]. 
\]
There are deterministic, $M$-independent constants $c,C>0$ such that $\P[F_M]\le Ce^{-cM^2}$ for each $M>0$. The same holds if we replace $\mcl B_t$ by $\mcl B_{2t}$
\end{lemma}
\begin{proof}
By the reflection principle for Brownian motion,
\[	
\P[\tau_{M-1}\le 2]\le \P\!\left[ \inf_{t \in [0,2]} \mcl B_t  \le 1+2a-M \right] \le Ce^{-cM^2}
\] 
for some $c,C>0$ as in the statement of the lemma.
	
	It remains to control the probability of $F_M\cap \{\tau_{M-1}>2 \}$.  We assume that $M>1$ (so that the time $\tau'_{M-1}$ we define next is well-defined and satisfies $\tau'_{M-1}<\tau_M$ almost surely). Let $\tau'_{M-1}$ be the last time $t$ before $\tau_M$ such that $\mcl B_t-at=1-M$. Since $\phi$ is supported on $[0,1]$, on the event $F_M\cap \{\tau_{M-1}>2\}$ there must be a time $t\in [\tau'_{M-1}-1,\tau'_{M-1}]$ such that $\mcl B_t-at\ge 0$. Note that the time reversal of $\{\mcl B_t-at: t\in [\tau'_{M-1}-1,\tau'_{M-1}] \}$ is a Brownian motion with drift starting from $1-M$ conditioned on the uniformly positive probability event that it does not reach $-M$ before time $1$. Hence Doob's maximal inequality implies $\P[F_M,\tau_{M-1}>2 ]\le Ce^{-cM^2}$. 
	
	By scaling, the statement still holds if we replace $\mcl B_{t}$ by $\mcl B_{2t}$.
\end{proof}

\begin{proof}[Proof of the lower bound of $\sigma(\gamma)$ given Proposition~\ref{prop:wedge-smooth}]
Given $\delta >0$, set $M=|\log\delta|^{\frac23}$. Let $A_M$ be the event that  
\[\inf \left\{t\in \R: \int_{-\infty}^\infty X^C_s\phi(s-t)ds=0 \right\}>\inf \left\{t\in \R: X^C_t=-M\right\}\]
where $X^C_t$ is the average process of $\hwedge^C$ in $(\cS,+\infty,-\infty)$. In this case, $X^C_t = \mcl B_{2t} - a t$ for $t\geq 0$, where $\mcl B$ is a standard linear Brownian motion and $a=Q-\frac32\gamma >0$. Furthermore, $A_M\subset F_M$ where $F_M$ is as in Lemma~\ref{lem:bump-circle} for this Brownian motion with drift. Therefore, $\P[A_M]\le C\exp(-c|\log\delta|^{\frac43}) = o_\delta^\infty(\delta)$.

On the event $G\cap \wt E_\delta^1 \cap A^c_M$, under $(\cS,+\infty,-\infty)$ coordinates and the smooth centering embedding of $\mcl W$, the following are true:
\begin{enumerate}[(i)]
\item $\wt\eta'([0,1])  \not\subset [-\log r,\infty)\times [0,\pi]$. \label{item:GEA-eta}
\item $\inf \{t\in \R: X^S_t=-M\}>0$.
\label{item:GEA-inf}
\item 	$\nu_{\hwedge^S}\!\left(\wt\eta'([0,1]) \cap (\R\times\{0\}) \right)\le \delta$ and $ \nu_{\hwedge^S}\!\left(\wt\eta'([0,1]) \cap (\R\times\{\pi\}) \right) \leq \delta$.
\label{item:GEA-length}
\end{enumerate}

\begin{figure}[ht!]
\begin{center}
\includegraphics[scale=1]{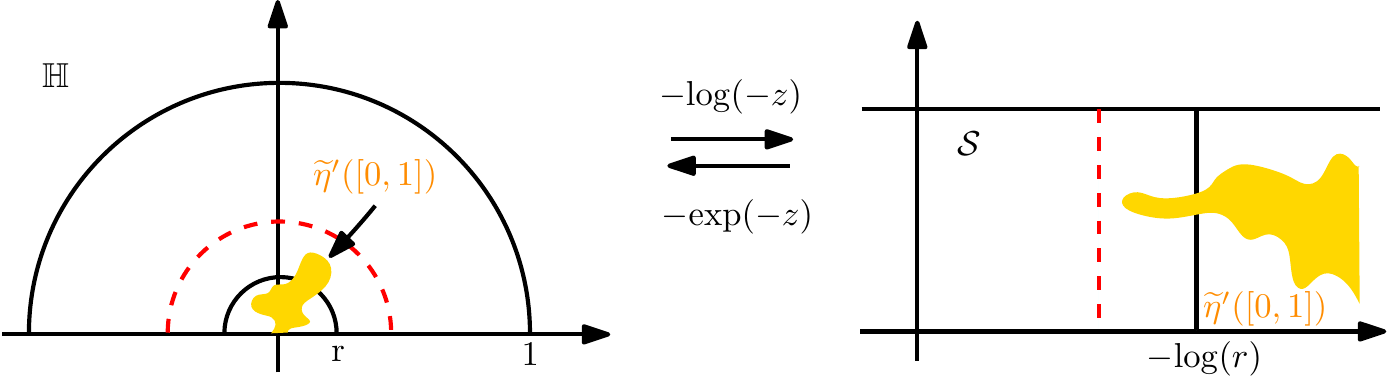}
\end{center}
\caption{Smooth centering embedding of $\mcl W$ into $(\BB H,0,\infty)$ (left) and $(\mcl S,+\infty,-\infty)$ (right). On the left (resp.\ right), the dotted red semi-circle (resp.\ line segment) corresponds to the unit radius (resp.\ intersection with the imaginary axis) under the circle average embedding of the modified field $\hwedge_M$. Note that this semi-circle (resp.\ line segment) is contained in (resp.\ to the right of) the unit circle (resp.\ imaginary axis) for the smooth centering embedding on the event $A_M^c$ in the proof of the lower bound of $\sigma(\gamma)$. The event $G$ of Proposition~\ref{prop:wedge-smooth} is such that if $G\cap \wt{E}^1_\delta$ occurs we have $\op{diam}(\wt{\eta}'([0,1]))>r$ on the left (resp.\ $\wt{\eta}'([0,1])$ is not contained in $[-\op{log} r,\infty)\times [0,\pi]$ on the right).}
\label{fig2}
\end{figure}

Let $\hwedge_M := \hwedge+M$. Since the law of a quantum wedge is invariant under multiplying its area by a constant~\cite[Proposition~4.6]{wedges}, $(\cS, \hwedge_M , +\infty ,-\infty)$ has the law of a $\frac32\gamma$-quantum wedge. Let $\hwedge_M^C$ be the circle-average embedding of this wedge into $\cS$. Let $\wt\eta'_M$ be given by $\eta'$ in $(\cS , +\infty,-\infty)$-coordinates parameterized by quantum mass with respect to $\hwedge_M^C$.
If $G\cap \wt E_\delta^1 \cap A^c_M$ occurs, then
if we consider $\hwedge^C_M$  under $(\cS,+\infty,-\infty)$ coordinates,  the above conditions~\eqref{item:GEA-eta} through~\eqref{item:GEA-length} imply that the following are true.
\begin{enumerate}[(i)]
\item $\wt\eta_M'([0,e^{\gamma M}])  \not\subset [-\log r,\infty)\times [0,\pi] $. 
\item 	$\nu_{\hwedge^C_M}\!\left(\wt\eta_M'([0,e^{\gamma M}]) \cap (\R\times\{0\}) \right)\le \delta e^{\gamma M/2}$ and\\
 $ \nu_{\hwedge^C_M}\!\left(\wt\eta_M'([0,e^{\gamma M}]) \cap (\R\times\{\pi\}) \right) \leq \delta e^{\gamma M/2}$.
\end{enumerate} 
In particular, if we switch back to $(\bbH,0,\infty)$ coordinates, the event $E^{T_r}_{\delta e^{\gamma M/2 }}$ as defined in Proposition~\ref{prop:quantum-exponent} with $(\hwedge^C_M,\wt\eta_M')$  in place of $(\hwedge^C,\wt\eta')$ occurs. Since $(\hwedge^C_M,\wt\eta_M')\overset{d}{=} (\hwedge^C,\wt\eta')$,
\[
\P[G\cap \wt E^1_\delta \cap A^c_M]\le \P[E^{T_r}_{\delta e^{\gamma M/2 }}].
\]
By Proposition~\ref{prop:quantum-exponent} (recall that $M=|\log\delta|^{\frac23}$) we have
\[
\lim_{\delta\to 0}\frac{\log\P[E^{T_r}_{\delta e^{\gamma M/2 }} ]}{\log\delta^{-1}}=-\frac{4}{\gamma^2}.
\]
By condition~\eqref{item:G-prob} in Proposition~\ref{prop:wedge-smooth},
\[ 
-\sigma(\gamma)=\lim_{\delta\to 0}\frac{\log\P[G,\wt E^1_\delta]}{\log\delta^{-1}}\le \lim_{\delta\to 0}\frac{\log\P[E^{T_r}_{\delta e^{\gamma M/2 }} ]}{\log\delta^{-1}}\vee
\lim_{\delta\to 0}\frac{\log\P[A_M ]}{\log\delta^{-1}}= -\frac{4}{\gamma^2}. \qedhere
\]
\end{proof} 

\subsubsection{Proof of Proposition~\ref{prop:wedge-smooth}}
\label{subsec:wedge-smooth}
In light of the preceding two subsections, to complete the proof of~\eqref{eq:exponent-compare} and hence of Theorem~\ref{thm:main}, it remains only to prove Proposition~\ref{prop:wedge-smooth}.
Recall the definition of the wedge $\mcl W_*$ from Section~\ref{subsub:notation}. We first construct an event where the scaling constant $\sigma_{\Psi , C}(\mcl W_*)$ (as defined in Section~\ref{subsub:notation}) is bounded from above and below. 
\begin{lemma}\label{lem:distortion}
	There is a deterministic constant $c_0 \in(0,1)$, an event $G_1$ which is measurable with respect to $(Z_t)_{t\in[0,1]}$, and an event $G_2$ which is independent of $(Z_t)_{t\in[0,1]}$ such that the following holds for each $\delta\in (0,\frac12)$:
	\begin{enumerate}[(i)]
	\item  $\P[G_1 \,|\,  \wt E^1_\delta ]\geq c_0$ and $\P[G_2]\geq c_0$; \label{item:distortion-prob}
	\item On the event $\wt E^1_\delta \cap G_1\cap G_2$, $\sigma_{\Psi, C}^{\bbH}(\cW_*)  \in [10^{-3},10^3]$\label{item:distortion-scaling}
	\end{enumerate}
\end{lemma}
\begin{proof}
	Let $G_1 :=\{1\le L_1,R_1\le 2\} $. By~\cite[Theorem~2]{shimura-cone}, $\P[G_1\,|\, \wt E_\delta^1 ]\geq c_0$ for some $c_0>0$ independent of $\delta$. 
	
	Suppose $\mcl W$ has the unit radius embedding into $(\bbH,0,\infty)$. Let $x^-$ and $x^+$ be defined such that $\wt{\eta}'([0,1]) \cap \R=[x^-,x^+]$. By \cite[Equation (3.14)]{lawler-book} and the definitions of the unit radius embedding and the $\Psi$-embedding, we have $|\Psi(x^+)-\Psi(x^-)|\in [ 10^{-2},10^2]$. 
On the other hand, if $\delta\in (0,\frac12)$ then on $G_1\cap  \wt E_\delta^1$, we have $\nu_{\hwedge_*^\Psi}([0,\Psi(x^+)]),\nu_{\hwedge_*^\Psi}([\Psi(x^-),0])\in[\frac12,3]$.

Let $$G_2:=\{\nu_{\hwedge^C_*}([-1,1])<1/2 \} \cap \{\nu_{\hwedge^C_*}([0,2])\ge 3\} \cap \{ \nu_{\hwedge^C_*}([-2,0]) \ge 3  \}.$$ Then $G_2$ is independent of $(Z_t)_{t\in[0,1]}$ and $\P[G_2]\geq c_0$ for some (possibly smaller) $c_0>0$ independent of $\delta$. On $G_2$, the interval $[\Psi(x^-),\Psi(x^+)]$ after mapping to the circle average embedding of $\cW_*$ will contain $[-1,1]$ and be contained in $[-3,3]$. Therefore, on $ \wt E^1_\delta \cap G_1\cap G_2$, the scaling factor between $\hwedge^\Psi_*$ and $\hwedge^C_*$ lies in $[10^{-3},10^3]$.
\end{proof}

	The following lemma is a variant of \cite[Proposition~10.19]{wedges}. The proof follows from essentially the same argument, so we will only give a very brief sketch.
\begin{lemma}
\label{lem:smooth-positive}
	Let $K$ be a fixed constant and $\wt \phi : \BB H \rta [0,\infty)$ be a radially symmetric smooth function supported on $B_1(0)\setminus B_{e^{-1}}(0)$.
	 Suppose $\hwedge^C$ is the circle average embedding of a $\frac32\gamma$-quantum wedge into $(\BB H,0,\infty)$. For $t\in \R$, let $\hwedge^C_t=\hwedge^C(e^t\cdot)+Q\log |e^t\cdot| $. Let $\cG$ be the collection of conformal maps of the following form: $g:\bbH\setminus A\rta \bbH$, where $A$ ranges over all hulls with a tip $p\in \partial A\setminus \R$ in $\ol{\BB H}$ such that $0\in A\subset \ol{B_{r}(0)}$ for some $r\in [0,1]$, and $g$ satisfies $g(p)=0,g(\infty)=\infty$, and $\lim_{z\to\infty}g(z)/z\in [10^{-3},10^3]$. For $t_0\in \R$, let $F(t_0)$   be the event that the inner product
	$(\hwedge_t^{C},|(g^{-1})'|^2 \wt\phi \circ g^{-1} )$ is bigger than  $K$ for all $t\ge t_0$ and $g\in\cG$. Then $\lim_{t_0 \to\infty}\P[F(t_0)]=1$.
\end{lemma} 
\begin{proof}
 Let $m_t :=\inf\{(\hwedge_t,|(g^{-1})'|^2\phi \circ g^{-1} ): g\in \cG \}$. By results from Gaussian analysis (see \cite[Proposition~10.18]{wedges} and the discussion afterwards), 
the random variable $\inf\{m_t:t\in [0,1] \}$
 has finite variance.
 Since $m_t$ has stationary increments, the Birkhoff ergodic theorem implies
$\lim_{t\to\infty}m_t=\infty$ $\textrm{a.s.}$
\end{proof}

\begin{figure}[ht!]
\begin{center}
\includegraphics[scale=1]{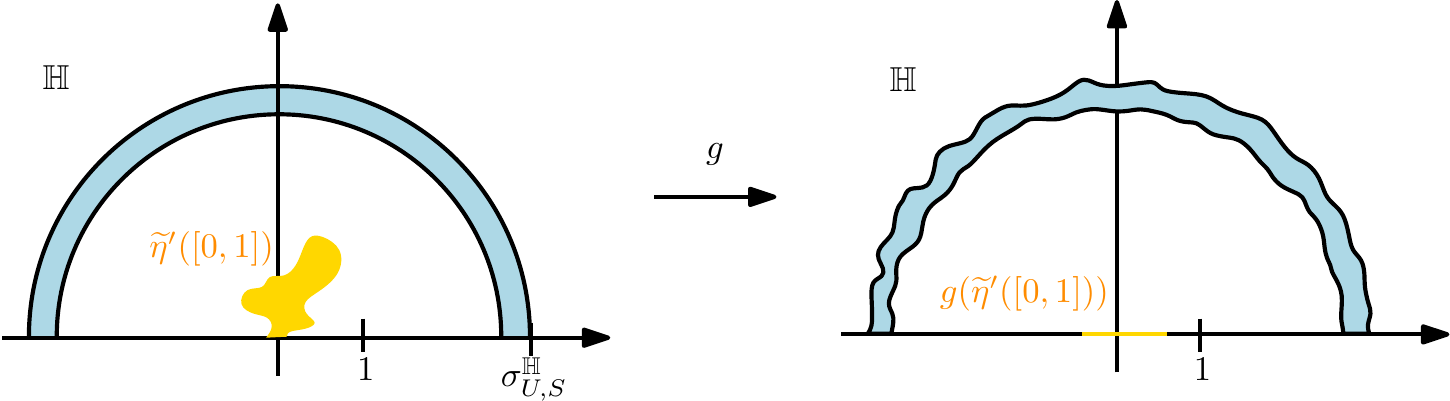}
\end{center}
\caption{Unit radius embedding of $\mcl W$ (left) and circle average embedding of $\mcl W_*$ (right). The boundary of $\wt{\eta}'([0,1])$ on the left divides $\mcl W$ into two independent quantum surfaces: $\mcl U = (\hwedge|_{\wt{\eta}'([0,1])},\wt{\eta}'([0,1]),0,\wt{\eta}'(1),\inf (\wt{\eta}'([0,1])\cap\R), \sup (\wt{\eta}'([0,1])\cap\R))$ 
and $\mcl W_* = (\hwedge|_{\wt{\eta}'([1,\infty))},\wt{\eta}'([1,\infty)),\wt{\eta}'(1),\infty)$. The occurrence of $\wt E_\delta^1$ depends only on $\mcl U$, while the diameter under the smooth centering embedding of $\wt{\eta}'([0,1])$ depends mainly on $\mcl W_*$. We use the independence of $\mcl U, \mcl W_*$ to establish Lemma~\ref{lem:distortion}, which implies that $g\in\mcl G$ on the event $G_1\cap G_2$ (see the statement of Lemmas~\ref{lem:distortion} and~\ref{lem:smooth-positive} for the notation). Then we approximate the smoothed drifted circle average for large radii on the left figure (in blue) by a 'distorted' average over the corresponding region on the right figure, and use Lemma~\ref{lem:smooth-positive} to conclude that for sufficiently large radii this is positive with uniformly positive probability conditioned on $\wt E^1_\delta$. This result is the content of Proposition~\ref{prop:wedge-smooth}.} 
\label{fig1}
\end{figure}

\begin{proof}[Proof of Proposition~\ref{prop:wedge-smooth}]
	Let $G_1,G_2 $, and $c_0$ be chosen so that the conclusion of Lemma~\ref{lem:distortion} holds. Throughout the proof we will assume $G_1\cap G_2\cap \wt E^1_\delta$ occurs.

Suppose $( \hwedge^U  , \wt\eta')$ is the unit radius embedding  of $(\mcl W , \wt \eta')$ into $(\bbH,0,\infty)$. Note that $\wt\eta'([0,1]) \subset \ol {B_1(0)}$ in this embedding. Let $g: \bbH\setminus \wt{\eta}'([0,1])  \rta  \bbH $  be such that $g(\wt\eta(1)) = 0$, $g(\infty) = \infty$ and 
\begin{equation}
\label{eq:UCchange}
\hwedge^U\circ g^{-1} + Q\log |(g^{-1})'|=\hwedge^C_*,
\end{equation}
where $\hwedge^C_*$ is the circle average embedding of $\mcl W_*$ into $(\bbH,0,\infty)$. By condition~\eqref{item:distortion-scaling} in Lemma~\ref{lem:distortion}, we have $g\in \cG$ where $\cG$ is defined as in Lemma~\ref{lem:smooth-positive}. 

Let $\phi$ be as in the definition of the smooth centering embedding. 
Let $\wt\phi :\BB H\mapsto[0,\infty)$ be the radially symmetric bump function on $\BB H$ such that $\pi\wt\phi(e^{-t})e^{-2t} = \phi(t)$ for each $t\in \R$. Then using polar coordinates and the canonical transformation between~$\bbH$ and~$\cS$, we have for all $t\in\R$ that
\eqb \label{eq:X-to-prod}
(\hwedge^U(e^t \cdot)+Q\log|e^t \cdot|,\wt\phi) = \int_{-\infty}^\infty X^U_s \phi (s+t) ds,
\eqe
where $(X^U_s)_{s\in \R}$ is the average process of $\hwedge^U$ over vertical lines in  $(\cS,+\infty,-\infty)$ coordinates. Define $\wt \phi_t(\cdot) :=e^{-2t}\wt\phi(e^{-t}\cdot)$. Note that $\wt \phi_t$ is a radially symmetric function on $\bbH$. By the coordinate change formula for the inner product $(\cdot,\cdot)$,
\begin{align*}
&(\hwedge^U(e^t \cdot)+Q\log|e^t \cdot|,\wt\phi)=(\hwedge^U(\cdot)+Q\log|\cdot|,\wt\phi_t) \notag \\
&=(\hwedge^U(g^{-1}(\cdot))+Q\log|g^{-1}(\cdot)|, 
|(g^{-1})'(\cdot)|^2 \wt\phi_t\circ g^{-1})
\end{align*}
Recalling~\eqref{eq:UCchange}, this is equal to
\begin{align} \label{eq:prod2}
&(\hwedge^C_* -Q\log|(g^{-1})'| +Q\log|g^{-1}(\cdot)|, 
|(g^{-1})'(\cdot)|^2 \wt\phi_t\circ g^{-1}) \notag \\
&=(\hwedge^C_*+Q\log|\cdot|, 
|g^{-1}(\cdot)|^2 \wt\phi_t\circ g^{-1})\notag \\
&+(-Q\log|(g^{-1})'(\cdot)| +Q\log|g^{-1}(\cdot)|-Q\log|\cdot|, 
|(g^{-1})'(\cdot)|^2 \wt\phi_t\circ g^{-1}).
\end{align}
By the Taylor expansion of $g$ near $\infty$, $|g(z)/z|$ is bounded from above and below outside of $\bbH\setminus B_{e^2}(0)$. Therefore there exists a constant $K_1 > 0$ such that for all $t\ge 2$ we have
\[ \big| (\log|g^{-1}(\cdot)|-\log|\cdot|, 
|(g^{-1})'(\cdot)|^2 \wt\phi_t\circ g^{-1}) \big |  = \big |(\log|\cdot|-\log|g(\cdot)|, \wt\phi_t) \big | \le K_1.\]
On the other hand, since $g\in \cG$, distortion estimates imply that $|g'(\cdot)|$ has universal upper and lower bounds on $\bbH\setminus B_{e^2}(0)$. Since $|(g^{-1})'(g(\cdot))|= 1/|g'(\cdot)|$, there exists a constant $K_2 > 0$ such that for all $t\ge 2$ we have
\[ \big | (\log|(g^{-1})'(\cdot)|, 
|(g^{-1})'(\cdot)|^2 \wt\phi_t\circ g^{-1}) \big | = \big |(\log|(g^{-1})'\circ g(
\cdot)|,  \wt\phi_t) \big | \le K_2 (1,\wt\phi_t)=K_2.\]
Thus the absolute value of the second term on the right in~\eqref{eq:prod2} is at most some deterministic constant $K$ for all  $t\ge 2$.
 
Let $F(t_0)$ be defined as in Lemma~\ref{lem:smooth-positive} for the quantum wedge $\cW_*$ and this choice of constant $K$.  Then
we can pick a deterministic $t_0 \geq 2$ large enough such that $\P[F(t_0)]\ge 1-c_0/2$.
 
Let $G:=G_1\cap G_2\cap F(t_0)$.
By independence of $\{G_1 ,  \wt E_\delta^1\}$ and $\{G_2,F(t_0)\}$ together with condition~\eqref{item:distortion-prob} in Lemma~\ref{lem:distortion}, we have $$\P[ G \,|\, \wt E^1_\delta] \geq c_0^2/2.$$
For $t\geq  0$, let $g_t(z):=e^{-t}g(e^tz)$ so that $g_t : \BB H\setminus e^{-t} \eta'([0,1]) \rta \BB H$. Then for $t\ge 0$, we have $g_t\in \cG$ and 
\eqbn
(\hwedge^C_*+Q\log|\cdot|, |(g^{-1})'(\cdot)|^2 \wt\phi_t\circ g^{-1})
=(\hwedge^C_*(e^t\cdot)+Q\log|e^t\cdot|, |(g_t^{-1})'(\cdot)|^2 \wt\phi\circ g_t^{-1}).
\eqen 
By definition of $F(t_0)$, on $G$ this latter quantity is at least $K$ for each $t\geq t_0$. 
By combining this with~\eqref{eq:X-to-prod} and~\eqref{eq:prod2}, we obtain
\[
\int_{-\infty}^\infty X^U_s \phi (s+t) ds \geq 0 \quad \forall t\ge t_0 .
\]
Therefore $\sigma_{S,U}^{\cS}(\cW)  \le t_0$, which means  $\sigma_{S,U}^{\bbH}(\cW)\ge e^{-t_0}$.  In light of Remark~\ref{rmk:ratio}, the constants $c = c_0^2/2$ and $r=e^{-t_0}$ and the event $G$ meet the requirements in Proposition~\ref{prop:wedge-smooth}.
\end{proof}

\bibliography{cibibshort,cibib,ref}
\bibliographystyle{hmralphaabbrv}
\end{document}